\documentclass[twoside,english,headsepline]{scrartcl}
\pdfoutput=1
\usepackage[T1]{fontenc}
\usepackage[utf8]{inputenc}
\usepackage[a4paper]{geometry}
\usepackage{color}
\usepackage{babel}
\usepackage{varioref}
\usepackage{refstyle}
\usepackage{mathtools}
\usepackage{amsmath}
\usepackage{amsthm}
\usepackage{amssymb}
\usepackage{esint}
\usepackage[authoryear]{natbib}
\usepackage{lmodern}
\usepackage{stix}
\usepackage[kerning=true]{microtype}
\usepackage[cmyk]{xcolor}
\usepackage[unicode=true,
 bookmarks=true,bookmarksnumbered=false,bookmarksopen=false,
 breaklinks=false,pdfborder={0 0 0},backref=false,colorlinks=true,urlcolor=blue]
 {hyperref}

\hypersetup{pdftitle={Numerical investigations of non-uniqueness for the Navier--Stokes initial value problem in borderline spaces},
 pdfauthor={Julien Guillod \& Vladimír Šverák}}

\usepackage{scrlayer-scrpage}
\ohead{\pagemark}
\rehead{J. Guillod \& V. Šverák}
\lohead{Numerical investigations of non-uniqueness for the Navier–Stokes initial value problem}
\ofoot{}
\makeatletter

\usepackage{doi}

\usepackage{graphics}
\newcommand{\includefigure}[2][]{\centering\includegraphics[#1]{figures/#2}}
\setcapindent{0pt}

\theoremstyle{plain}
\newtheorem{thm}{Theorem}
\newtheorem{prop}{Proposition}

\newtheorem*{result}{Numerical Observations}
\theoremstyle{remark}

\global\long\def\e{\mathrm{e}}
\global\long\def\rd{\mathrm{d}}
\global\long\def\bu{\boldsymbol{u}}
\global\long\def\bv{\boldsymbol{v}}
\global\long\def\bw{\boldsymbol{w}}
\global\long\def\bomega{\boldsymbol{\omega}}
\global\long\def\bA{\boldsymbol{A}}
\global\long\def\bx{\boldsymbol{x}}
\global\long\def\be{\boldsymbol{e}}

\global\long\def\ba{\boldsymbol{a}}

\global\long\def\bphi{\boldsymbol{\phi}}
\global\long\def\bvphi{\boldsymbol{\varphi}}
\global\long\def\bpsi{\boldsymbol{\psi}}
\global\long\def\btau{\boldsymbol{\tau}}
\global\long\def\bU{\boldsymbol{U}}
\global\long\def\bV{\boldsymbol{V}}
\global\long\def\bUsigma{\bU_{\!\sigma}}
\global\long\def\bUzero{\bU_{\!0}}
\global\long\def\bAzero{\bA_{0}}

\global\long\def\bVsigma{\bV_{\!\!\sigma}}
\global\long\def\bTsigma{\boldsymbol{T}_{\!\sigma}}
\global\long\def\bOmegasigma{\boldsymbol{\Omega}_{\sigma}}
\global\long\def\bnabla{\boldsymbol{\nabla}}
\global\long\def\bcdot{\boldsymbol{\cdot}}
\global\long\def\bwedge{\boldsymbol{\wedge}}
\global\long\def\bzero{\boldsymbol{0}}

\global\long\def\axi{\mathrm{axi}}
\global\long\def\Re{\operatorname{Re}}

\global\long\def\Range{\operatorname{Range}}
\global\long\def\Kernel{\operatorname{Kernel}}
\global\long\def\hypF{\,_{1\!}F_{1}}
\global\long\def\scale{\varkappa}
\global\long\def\cL{\mathcal{L}}
\global\long\def\proj{\mathcal{P}\!}
\global\long\def\bZ{\mathbb{Z}}

\SetExtraKerning[unit=space]
    {encoding={*}, family={stix}, series={*}, size={footnotesize,small,normalsize}}
    {\textendash={400,400}}

\makeatletter
\let\amstexbig\big
\def\newbig#1{%
  \ifx#1|%
    \expandafter\@firstoftwo
  \else
    \expandafter\@secondoftwo
  \fi
  {\big@bar}%
  {\amstexbig{#1}}%
}
\AtBeginDocument{\let\big\newbig}
\def\big@bar{\bBigg@{1.1}|}
\makeatother

\begin{document}

\title{Numerical investigations of non-uniqueness\\for the Navier–Stokes initial value problem \\in borderline spaces}

\author{Julien Guillod\textsuperscript{1} \qquad{} Vladimír Šverák\textsuperscript{2}\\
{\small{}\textsuperscript{1}UFR de Mathématiques, Université Paris-Diderot}\\
{\small{}\textsuperscript{2}School of Mathematics, University of Minnesota}}

\maketitle
\begin{abstract}
We consider the Cauchy problem for the incompressible Navier–Stokes equations in $\mathbb{R}^{3}$ for a one-parameter family of explicit scale-invariant axi-symmetric initial data, which is smooth away from the origin and invariant under the reflection with respect to the $xy$-plane. Working in the class of axi-symmetric fields, we calculate numerically scale-invariant solutions of the Cauchy problem in terms of their profile functions, which are smooth. The solutions are necessarily unique for small data, but for large data we observe a breaking of the reflection symmetry of the initial data through a pitchfork-type bifurcation. By a variation of previous results by \citet{Jia-Localinspace2013} it is known rigorously that if the behavior seen here numerically can be proved, optimal non-uniqueness examples for the Cauchy problem can be established, and two different solutions can exists for the same initial datum which is divergence-free, smooth away from the origin, compactly supported, and locally $(-1)$-homogeneous near the origin. In particular, assuming our (finite-dimensional) numerics represents faithfully the behavior of the full (infinite-dimensional) system, the problem of uniqueness of the Leray–Hopf solutions (with non-smooth initial data) has a negative answer and, in addition, the perturbative arguments such those by \citet{Kato1984} and \citet{Koch-Tataru2001}, or the weak-strong uniqueness results by Leray, Prodi, Serrin, Ladyzhenskaya and others, already give essentially optimal results.  There are no singularities involved in the numerics, as we work only with smooth profile functions. It is conceivable that our calculations could be upgraded to a computer-assisted proof, although this would involve a substantial amount of additional work and calculations, including a much more detailed analysis of the asymptotic expansions of the solutions at large distances.
\end{abstract}
\textsf{\textbf{MSC classes}}\quad{}35A02, 35D30, 35Q30, 76D05, 76D03, 76M10

\newpage

\section{Introduction}

We consider the Cauchy problem for the Navier–Stokes equations in
$(0,\infty)\times\mathbb{R}^{3}$,
\begin{align}
\partial_{t}\bu+\bu\bcdot\bnabla\bu & =\Delta\bu-\bnabla p\,, & \bnabla\bcdot\bu & =0\,, & \bu(0,\cdot) & =\bu_{0}\,.\label{eq:ns-cauchy}
\end{align}
Important open questions about the Cauchy problem \eqref{ns-cauchy}
concern existence and uniqueness of the solutions in suitable classes of functions.
There are essentially two methods to address
these issues: the global method based on \emph{a
priori} energy estimates and the local  perturbation
theory method.

The global method started by the seminal work of \citet{Leray-Surlemouvement1934}
and further developed by \citet{Hopf-UeberdieAnfangswertaufgabe1950},
has lead to the concept of the Leray–Hopf solutions. For $\bu_{0}\in L^{2}(\mathbb{R}^{3})$,
a Leray–Hopf  solution of the Cauchy problem in $(0,T)\times\mathbb{R}^{3}$ is a field
$\bu\in L^{\infty}(0,T;L^{2}(\mathbb{R}^{3}))\cap L^{2}(0,T;\dot{H}^{1}(\mathbb{R}^{3}))$
satisfying \eqref{ns-cauchy} weakly, together with  some additional requirements, such as the energy inequality
\[
\left\Vert \bu(t,\cdot)\right\Vert _{L^{2}(\mathbb{R}^{3})}^{2}+2\int_{0}^{t}\left\Vert \bnabla\bu(\tau,\cdot)\right\Vert _{L^{2}(\mathbb{R}^{3})}^{2}\rd\tau\leq\left\Vert \bu(0,\cdot)\right\Vert _{L^{2}(\mathbb{R}^{3})}^{2}\,,
\]
for all $t\in(0,T)$, and suitable continuity of the map $t\to \bu(t)$. We will be dealing with solutions which are smooth for $t>0$, and  have only a weak singularity at $(x,t)=(0,0)$, so the exact technical assumptions are not our focus here, although in connection with the scale-invariant solutions one should  point out the  important generalization by \citet{Lemarie-Rieusset2002, Lemarie-Rieusset2016}, where the global energy requirements are replaced by local ones.

By using the energy estimate and compactness
arguments, Leray and others showed the existence of a global Leray–Hopf solution for any
$\bu_{0}\in L^{2}(\mathbb{R}^{3})$ with $T=\infty$.  However, the proof, which relies on compactness arguments, does not give anything concerning uniqueness, except in the case when existence of a more regular solution is known. In that case one can use the energy arguments around the more regular solution, and show that any Leray–Hopf solution has to coincide with the regular one. These results, now known as weak-strong uniqueness theorems, go back to \citet{Leray-Surlemouvement1934}, with later generalizations by a number of authors, including for example \citet{Prodi-uniqueness1959} and \citet{Serrin-initialvalueproblem1963}. One has to mention also the results by \citet{Kiselev-Ladyzhenskaya1957}, where a slightly different approach is taken.

The perturbation method goes back to \citet{Oseen1911} and \citet{Leray-Surlemouvement1934} and was later developed in scale-invariant spaces by \citet{Kato-nonstationaryNavierStokes1962,Fujita-NavierStokesinitial1964} and \citet{Kato1984}.
It treats the  nonlinearity as a perturbation and inverts the linear part to obtain an integral equation, which is then approached via the Picard iteration. The borderline spaces for this method are scale-invariant with respect to the scaling symmetry of \eqref{ns-cauchy}:
\begin{equation}
\begin{aligned}\bu(t,x) & \to \bu_{\scale}(t,\bx) =\scale\bu(\scale^{2}t,\scale\bx)\,,\\
p(t,x) & \to p_{\scale}(t,\bx) =\scale^{2}p(\scale^{2}t,\scale\bx)\,,\\
\bu_0(x) & \to \bu_{0\scale}(\bx) =\scale\bu_{0}(\scale\bx)\,,
\end{aligned}
\label{eq:scaling}
\end{equation}
where $\scale>0$.
A space $X$ for the initial datum $\bu_0$ is scale-invariant if its norm is invariant under the scaling of the initial condition.
Well-known scale-invariant spaces $X$ relevant for the Cauchy problem
\eqref{ns-cauchy} are, for example, the spaces $L^3$, see \citet{Kato1984},
and ${\mathrm{BMO}}^{-1}$, see \citet{Koch-Tataru2001}. An important distinction
between the two is that the former does not contain the function $|\bx|^{-1}$,
whereas the latter does. This is related to the fact that for $L^3$ one can show
local-in-time well-posedness for data of any size (with the time of existence
depending on the datum), whereas for ${\mathrm{BMO}}^{-1}$ one can only treat small data.
The results in this paper suggest that this is not an artifact of the methods,
but reflects the actual behavior of the solutions.\\

A special class of initial data being naturally in ${\mathrm{BMO}}^{-1}$
and not in $L^3$ is given by scale-invariant initial data. An initial datum $\bu_0$ is
scale-invariant under the scaling symmetry \eqref{scaling} if $\bu_{0\scale}=\bu_{0}$
for all $\scale>0$. In particular such initial data behaves like $|\bx|^{-1}$ both near
the origin and at infinity, so are not in $L^3$. For scale-invariant initial data,
it is natural to look for the solutions of \eqref{ns-cauchy} as being also invariant
under the scaling \eqref{scaling} \emph{i.e.} satisfying $\bu_{\scale}=\bu$
and $p_{\scale}=p$ for all $\scale>0$. A scale-invariant solution $(\bu,p)$ has the form
\begin{align*}
\bu(t,\bx) & =\frac{1}{t^{1/2}}\bU\biggl(\frac{\bx}{t^{1/2}}\biggr)\,, & p(t,\bx) & =\frac{1}{t}P\biggl(\frac{\bx}{t^{1/2}}\biggr)\,,
\end{align*}
where the profiles $\bU=\bu(\cdot,1)$ and $P=p(\cdot,1)$ satisfy
\begin{subequations}
\begin{align}
\Delta\bU+\frac{\bx}{2}\bcdot\bnabla\bU+\frac{1}{2}\bU-\bU\bcdot\bnabla\bU-\bnabla P & =\bzero\,, & \bnabla\bcdot\bU & =0\,,\label{eq:ns-scale-eq}
\end{align}
in $\mathbb{R}^{3}$ together with the condition
\begin{equation}
\bU(\bx)=\bu_{0}(\bx)+o(\left|\bx\right|^{-1})\qquad\text{as}\qquad\left|\bx\right|\to\infty\,.\label{eq:ns-scale-bc}
\end{equation}
\label{eq:ns-scale}\end{subequations}
\citet{Jia-Localinspace2013} proved the following global existence result for scale-invariant initial data. A different proof was obtained by \citet{Bradshaw-Tsai2017}.
\begin{thm}
\label{thm:jia-sverak}If $\bu_{0}\in C^{\infty}(\mathbb{R}^{3}\setminus\{\bzero\})$
is scale-invariant and divergence-free, then there exists a least
one scale-invariant solution $\bu\in C^{\infty}((0,\infty)\times\mathbb{R}^{3})$
of \eqref{ns-cauchy}. Moreover if $\bu$ is a scale-invariant solution
then the profile $\bU=\bu(1,\cdot)$ satisfies \eqref{ns-scale-eq}
and
\begin{equation}
\bigl|\partial^{\alpha}\bigl(\bU(\bx)-\e^{\Delta}\bu_{0}(\bx)\bigr)\bigr|\leq\frac{C(\alpha,\bu_{0})}{\left(1+\left|\bx\right|\right)^{3+|\alpha|}}\,,\label{eq:bound-U}
\end{equation}
for any $\alpha$.
\end{thm}
The scale-invariant solution is unique for small initial data.
For large initial data it has been conjectured by \citet{Jia-Localinspace2013, Jia-Areincompressible3d2015}
that the scale-invariant solution is not unique.
Our goal is to present numerical evidence for this conjecture.
From the existence of such solutions, the non-uniqueness
of Leray–Hopf solutions and the sharpness of the Serrin uniqueness
criterion can be established along the lines of \citet{Jia-Areincompressible3d2015}.
In the case of the harmonic map
heat flow, related results have been obtained by \citet{Germain-flowmap2016}.\\

We now introduce the function spaces needed for the study of \eqref{ns-scale}.
Let
\begin{subequations}
\begin{equation}
\mathcal{U}=\bigl\{\bU\in L^{\infty}(\mathbb{R}^{3})\colon\bnabla\bcdot\bU=0\;\text{and}\;\left\Vert \bU\right\Vert _{\mathcal{U}}<\infty\bigr\}\,,\label{eq:def-U-space}
\end{equation}
where
\begin{equation}
\left\Vert \bU\right\Vert _{\mathcal{U}}=\sum_{\left|\alpha\right|\le2}\sup_{\bx\in\mathbb{R}^{3}}\left(1+\left|\bx\right|\right)^{1+\alpha}\left|\partial^{\alpha}\bU(\bx)\right|\,.\label{eq:def-U-norm}
\end{equation}
\label{eq:def-U}\end{subequations}
The profile $\bU$ of a scale-invariant solution belongs naturally to $\mathcal{U}$.
Let
\begin{subequations}
\begin{equation}
\mathcal{V}=\bigl\{\bv\in L^{2}(\mathbb{R}^{3})\cap L^{4}(\mathbb{R}^{3}):\bnabla\bcdot\bv=0\bigr\}\,,\label{eq:def-V-space}
\end{equation}
with the norm
\begin{equation}
\left\Vert \bv\right\Vert _{\mathcal{V}}=\left\Vert \bv\right\Vert _{L^{2}(\mathbb{R}^{3})}+\left\Vert \bv\right\Vert _{L^{4}(\mathbb{R}^{3})}\,.\label{eq:def-V-norm}
\end{equation}
\label{eq:def-V}\end{subequations}
In view of \eqref{bound-U}, the difference
between two different scale-invariant solutions sharing the same initial datum will be in $\mathcal{V}$.
We define $\mathcal{D}$ as the
following subspace of $\mathcal{V}$,
\begin{subequations}
\begin{equation}
\mathcal{D}=\bigl\{\bv\in\mathcal{V}:\partial^{\alpha}\bv\in\mathcal{V}\:\text{and}\:\bx\bcdot\bnabla\bv\in\mathcal{V}\:\text{for}\:0\leq\left|\alpha\right|\leq2\bigr\}\,,\label{eq:def-D-space}
\end{equation}
with the norm
\begin{equation}
\left\Vert \bv\right\Vert _{\mathcal{D}}=\left\Vert \bv\right\Vert _{\mathcal{V}}+\left\Vert \bnabla\bv\right\Vert _{\mathcal{V}}+\Vert\bnabla^{2}\bv\Vert_{\mathcal{V}}+\left\Vert \bx\bcdot\bnabla\bv\right\Vert _{\mathcal{V}}\,.\label{eq:def-D-norm}
\end{equation}
\label{eq:def-D}\end{subequations}
The subspace of axi-symmetric vector fields in $\mathcal{D}$ is denoted by $\mathcal{D}_{\axi}$.
Given some fixed scale-invariant vector-field $\ba_{0}\in C^{\infty}(\mathbb{R}^{3}\setminus\left\{ \bzero\right\} )$,
we define the map $F\colon\mathcal{D}\times\mathbb{R}\to\mathcal{V}$ by
\begin{equation}
F(\bv,\sigma)=-\Delta\bv-\frac{\bx}{2}\bcdot\bnabla\bv-\frac{1}{2}\bv+\left(\sigma\bAzero+\bv\right)\bcdot\bnabla\left(\sigma\bAzero+\bv\right)+\bnabla P\,,\label{eq:def-F}
\end{equation}
where $\bAzero=\e^{\Delta}\ba_{0}$ and $P$ is chosen such that $F(\bv,\sigma)$
is divergence-free. Therefore $\bU$ is a solution of \eqref{ns-scale}
with $\bu_{0}=\sigma\ba_{0}$ for $\sigma\in\mathbb{R}$ if and only
if $F(\bv,\sigma)=\bzero$ where $\bv=\bU-\sigma\bAzero$. The linearization
of \eqref{ns-scale} around $\bU\in\mathcal{U}$ is defined as
the operator $\mathcal{L}(\bU)\colon\mathcal{D}\to\mathcal{V}$ where
\begin{equation}
\cL(\bU)\bphi=-\Delta\bphi-\frac{\bx}{2}\bcdot\bnabla\bphi-\frac{1}{2}\bphi+\bU\bcdot\bnabla\bphi+\bphi\bcdot\bnabla\bU+\bnabla p\,,\label{eq:def-L}
\end{equation}
and in particular $D_{1}F(\bv,\sigma)=\cL(\sigma\bAzero+\bv)$. The operator $\mathcal{L}(\bU)$ is viewed as an unbounded operator in $\mathcal{V}$ with domain $\mathcal{D}$.\\

For small values of $\sigma$ the solution of $F(\bv,\sigma)=\bzero$
is unique, leading to a unique solution $\bUsigma=\sigma\bAzero+\bv$
of \eqref{ns-scale}. As long as the kernel of the linearization $D_{1}F(\bv,\sigma)=\cL(\bUsigma)$
is trivial, the solution can be (locally) uniquely continued  to larger values
of $\sigma$. However, if at some value of $\sigma=\sigma_{0}$ this
kernel is no more trivial then another solution can bifurcate from
$\bv$ leading to the non-uniqueness of solutions of \eqref{ns-scale}.

The following result on the spectrum of the linearization
$\cL(\bU)$ follows essentially by the results of \citet{Gallay.Wayne-Invariantmanifoldsand2002}
and \citet{Jia-Areincompressible3d2015}.
\begin{thm}
\label{thm:spectrum-L}
We have:
\begin{enumerate}
\item The spectrum of $\cL(\bzero)$ is given by
\[
\sigma\bigl(\cL(\bzero)\bigr)=\bigl\{\lambda\in\mathbb{C}:\Re\lambda\geq\tfrac{3}{4}\bigr\}\cup\bigl\{\tfrac{3}{2}+n\,,\;n\in\mathbb{N}\bigr\}\,.
\]
The eigenvectors corresponding to the continuous part $\Re\lambda>\tfrac{3}{4}$
decay at infinity like $\left|\bx\right|^{-2\lambda}$, whereas the
eigenvectors corresponding to $\tfrac{3}{2}+n$ decay exponentially
fast like $\e^{-\left|\bx\right|^{2}/4}$. The multiplicity of the
eigenvalue $\tfrac{3}{2}+n$ is $(n+1)(n+3)$ with domain $\mathcal{D}$ and $n+1$ with the axi-symmetric domain $\mathcal{D}_\axi$.
\item If $\bU\in\mathcal{U}$, the spectrum of $\cL(\bU)$ satisfies
\begin{equation}
\sigma\bigl(\cL(\bU)\bigr)\subset\bigl\{\lambda\in\mathbb{C}:\Re\lambda\geq\tfrac{3}{4}\bigr\}\cup S\,,\label{eq:spectrum-LU}
\end{equation}
where $S$ is a discrete set such that $\bigl\{\lambda\in S:\Re\lambda\leq\delta\bigr\}$
is finite for any $\delta<\frac{3}{4}$.
\end{enumerate}
\end{thm}

This theorem ensures that only discrete spectrum can cross the imaginary axis.
One can establish the following continuation and bifurcation results depending
on the behavior of the discrete spectrum:
\begin{thm}
\label{thm:continuation-bifurcation}
Let $\sigma_{0}\in\mathbb{R}$ and $\bv_{0}\in\mathcal{D}$ be a solution of $F(\bv_{0},\sigma_{0})=\bzero$,
so $\bUzero=\sigma_{0}\bAzero+\bv_{0}$ is a solution of \eqref{ns-scale}
with $\bu_{0}=\sigma_{0}\ba_{0}$.
\begin{enumerate}
\item If zero is not in the spectrum $\sigma(\cL(\bUzero))$, then there exist $\varepsilon>0$
and a unique smooth solution curve $\bv\colon(\sigma_{0}-\varepsilon,\sigma_{0}+\varepsilon)\to\mathcal{D}$
such that $F(\bv(\sigma),\sigma)=\bzero$ and $\bv(\sigma_{0})=\bv_{0}$.
In particular $\bUsigma=\sigma\bAzero+\bv(\sigma)$ is a solution
of \eqref{ns-scale} with $\bu_{0}=\sigma\ba_{0}$.
\item Assume the existence of a smooth solution curve $\bv_{1}:(\sigma_{0}-\varepsilon,\sigma_{0}+\varepsilon)\to\mathcal{D}$
such that $F(\bv_{1}(\sigma),\sigma)=\bzero$. If the spectrum of
the linearization has the form
\begin{equation}
\bigl\{0\bigr\}\subset\sigma\bigl(\cL(\bUzero)\bigr)\subset\bigl\{\lambda\in\mathbb{C}:\Re\lambda>\delta\bigr\}\cup\bigl\{0\bigr\}\,,
\label{eq:spectral-condition}
\end{equation}
for some $\delta>0$, where zero is a simple eigenvalue with associated
eigenvector $\bphi$ and if
\begin{equation}
\bpsi\bcdot\bnabla\bphi+\bphi\bcdot\bnabla\bpsi+\bnabla p\notin\Range\bigl(\cL(\bUzero)\bigr)\,,\label{eq:thm-bif-hyp}
\end{equation}
where $\bpsi=\bA_{0}+\bv_{1}^{\prime}(\sigma_{0})=\partial_{\sigma}\bUsigma\bigl|_{\sigma=\sigma_{0}}$,
then in addition to the solution curve $\bigl\{\bigl(\bv_{1}(\sigma),\sigma\bigr)\in\mathcal{D}\times\mathbb{R}\,,\;\sigma\in(\sigma_{0}-\varepsilon,\sigma_{0}+\varepsilon)\bigr\}$
through $(\bv_{0},\sigma_{0})$ there exists another unique smooth
solution curve $\bigl\{\bigl(\bv_{2}(s),\sigma_{2}(s)\bigr)\in\mathcal{D}\times\mathbb{R}\,,\;s\in(-\varepsilon,\varepsilon)\bigr\}$
through $(\bv_{0},\sigma_{0})$ such that $F(\bv_{2}(s),\sigma_{2}(s))=\bzero$,
$\bv_{2}(0)=\bv_{0}$, and $\sigma_{2}(0)=\sigma_{0}$.

If $\bphi\bcdot\bnabla\bphi+\bnabla p\notin\Range\bigl(\cL(\bUzero)\bigr)$
the bifurcation is transcritical.
If
$\bphi\bcdot\bnabla\bphi+\bnabla p\in\Range\bigl(\cL(\bUzero)\bigr)$ and an additional non-degeneracy assumption is satisfied, we are dealing with a pitchfork bifurcation.
\end{enumerate}
\end{thm}

Our aim is to choose a particular scale-invariant vector-field
$\ba_{0}$ and to construct numerical solutions exhibiting the bifurcation.
In the situation that we will consider here we will have an additional structure
coming from a $\bZ_2$-symmetry. The branch $\bv_1$ will correspond to the solutions
invariant under the $\bZ_2$-symmetry, whereas the branch $\bv_2$ will correspond
to the solutions with broken symmetry. The branch itself (as a set) will be invariant
under the symmetry, and hence the bifurcation will necessarily be of pitchfork type,
even though the usual non-degeneracy condition used for pitchfork bifurcations
may not be satisfied.\\

A possible mechanism behind the ill-posedness can be explained at a heuristic
level as follows. Assume the space is filled with an incompressible fluid and
consider a portion of the fluid of the shape of a thin disc
$\{x^2+y^2\le R^2\,,\,\,|z|\le \varepsilon\}$. If we impose on this portion of
the fluid a fast rotation about the $z$-axis with a smooth, albeit sharp
transition to zero outside of the disc, the centrifugal force will result in an
outward motion of the fluid along the $xy$-plane, which will be superimposed on
the rotation. One can perhaps compare this outward flux with a jet of fluid,
except that our ``jet'' goes out from the origin in all directions lying in the
$xy$-plane, rather than just in one direction. Due to incompressibility there
must also be an inward flux to the origin, which will take place along the
$z$-axis. Assume the velocity field is invariant under the reflection about the
$xy$-plane. When the field is large, the flow can be expected to be unstable,
and any slight deviation from the reflection symmetry will quickly lead to a
full breaking of this symmetry. To break the symmetry for smooth solutions, we
need an outside impulse. It may be very small, but it cannot be zero. However,
the situation may be different for non-smooth initial data. We can imagine a
scale-invariant initial datum $\bu_0$ which resembles a rotation localized in
the $xy$-plane as much as possible. This can be achieved only to a degree, since
the scale-invariance poses its own restrictions, which are of course not
compatible with fast decay (among other things). We can still think of $\bu_0$
as imposing significant rotation in some bounded region of the $xy$-plane, but
falling  off to zero as we move away from the plane. Also, $\bu_0$ is smooth
except at the origin, where it of course cannot be smooth due to the scale
invariance (as long as it is non-trivial).
Now the symmetry breaking impulse can essentially come from within the singular
point, and as such it can be completely hidden from the information provided by
the initial datum. In other words, at this level of singularity in the data, the
model is asked to operate based on the information which is insufficient for it,
somewhat similarly as in the non-uniqueness induced by reverse bubbling in the
2d harmonic map heat flow, see \cite{Topping2002}, for example. This analogy is
not perfect, as the 2d harmonic map heat flow is critical. For the 3d case we
refer the reader to \citet{Germain-flowmap2016} already quoted above.

The very general method used in perturbation theory or weak-strong uniqueness
breaks down exactly at this point. The numerics presented below suggests that,
at least when well-posedness for rough initial data is concerned, the non-linear
term in the equation does not seem to have any magical properties which would
enable one to go beyond the general perturbation analysis. We emphasize that
this conclusion may not apply to the problem of singularity formation from
smooth data.  The situation there may or may not be the similar (see for example
\citet{Tao2014}), but our results say nothing about it. However, if a
singularity is formed, our results suggest that, quite likely, uniqueness may be
lost. The connection between loss of regularity and uniqueness is, of course,
not new. Already in the 1950s, Ladyzhenskaya  emphasized the possibility of
non-uniqueness for solutions with insufficient regularity (including the
Leray–Hopf solutions), and \citet{Ladyzhenskaya-uniquenessandsmoothness1967}
presented an example closely related to the scenario discussed in this paper.\\

Due to our limited computational resources, we will work with axi-symmetric solutions,
\emph{i.e.} solutions which are invariant under the rotations around
the $z$-axis. In addition, we consider the $\bZ_2$-symmetry $\mathcal{R}$
defined by the reflection with respect to the plane $z=0$.
For the reasons previously explained, we choose the following scale-invariant
axi-symmetric divergence-free vector field for the initial data
\[
\ba_{0}(r,z)=\frac{\e^{-4(z/r)^{2}}}{\sqrt{r^{2}+z^{2}}}\be_{\theta}\,,
\]
where $(r,\theta,z)$ denote the cylindrical coordinates.
We note that this initial datum has ``pure swirl'' and is clearly invariant under
$\mathcal{R}$. In this paper we do not consider the breaking of the axial symmetry,
although it is conceivable that for some classes of the initial data this may occur.

The solutions we are dealing with in our work here are defined on the whole space
$\mathbb{R}^3$, and hence some truncation of the domain is needed for the numerics.
The solutions have good asymptotic expansions for $|\bx|\to\infty$,
which in principle could be calculated to a higher order precision.
However, the most obvious approximations seem to work quite well for the numerics,
and therefore we did not use the higher order expansions.
No doubt a possible computer-assisted proof would need to work with more
sophisticated approximations for large~$|\bx|$.

The numerical methods are described in \secref{methods} and our numerical results
are presented in \secref{results}, which can be summarized as follows:

\begin{result}
We numerically observe the following:
\begin{enumerate}
\item\label{enu:steady}
In the range $\sigma\in[0,500]$, there exists
a smooth curve of axi-symmetric and $\mathcal{R}$-symmetric self-similar solution
$\bUsigma$ of \eqref{ns-scale}, with $\bUsigma(\bx)=\sigma \ba_0(\bx)+o(|\bx|^{-1})$ at infinity.

\item\label{enu:spectrum}
The spectrum of the linearization
$\cL(\bUsigma)$ with domain $\mathcal{D}_{\axi}$ has the form
\begin{equation}
\bigl\{\lambda_{\sigma}\bigr\}\subset\sigma\bigl(\cL\bigl(\bUsigma\bigr)\bigr)\subset\bigl\{\lambda\in\mathbb{C}:\Re\lambda>\delta\bigr\}\cup\bigl\{\lambda_{\sigma}\bigr\}\,,\label{eq:numerics-spectrum}
\end{equation}
for some $\delta>0$, and there exists $\sigma_{0}\approx292$ such that $\lambda_{\sigma}>0$ for $\sigma<\sigma_{0}$, $\lambda_{\sigma}=0$ for $\sigma=\sigma_{0}$, and $\lambda_{\sigma}<0$ for $\sigma>\sigma_{0}$.
Near $\sigma_0$, the eigenvalue $\lambda_{\sigma}$ is simple, continuous in $\sigma$ and the associated eigenvector is not $\mathcal{R}$-symmetric.

\item\label{enu:bifurcation}
At $\sigma=\sigma_{0}$ there is a supercritical
pitchfork-type bifurcation corresponding to the breaking of the symmetry
$\mathcal{R}$. More precisely, for $\sigma\in[\sigma_0,500]$, in
addition to $\bUsigma$, there exists two axi-symmetric solutions $\bUsigma+\bVsigma$
and $\bUsigma+\mathcal{R}\bVsigma$ of \eqref{ns-scale} where $\bVsigma=\bzero$
for $\sigma=\sigma_{0}$ and $\bVsigma$ is not $\mathcal{R}$-symmetric
(hence non trivial) for $\sigma>\sigma_{0}$.
\end{enumerate}
\end{result}
In particular, this suggests that the solutions of the Navier–Stokes equations
are not unique on any time-interval for large initial data in the Lorentz space
$L^{3,\infty}$. This would mean, that the smallness assumption required by
\citet[Theorem 8.2]{Lemarie-Rieusset2016} for proving the local well-posedness
for initial data in $L^{3,\infty}$ is not technical, but reflect the actual nature
of the equations. The same conclusion holds for the result by
\citet{Koch-Tataru2001} for initial data in ${\mathrm{BMO}}^{-1}$.\\

The scale-invariant solutions have infinite energy, however, by following
the ideas of \citet[Theorem 1.2]{Jia-Areincompressible3d2015},
the different self-similar solutions can be localized:

\begin{thm}
\label{thm:localization}
Assume that \eqref{ns-scale} exhibits the same solution behavior as observed
in the above reported numerical results. Then there exists $T>0$ and
two different axi-symmetric Leray–Hopf solutions of \eqref{ns-cauchy} on $(0,T)$
with the same compactly supported axi-symmetric initial datum
$\bu_{0}\in C^{\infty}(\mathbb{R}^{3}\setminus\{\bzero\})$
with $\bu_{0}(\bx)=O(|\bx|^{-1})$ near the origin. Moreover,
these two Leray–Hopf solutions are smooth for $t\in (0,T)$
and belong to $L^{p}(0,T;L^{q}(\mathbb{R}^{3}))$ for any $p,q$ with
\begin{equation}
\frac{2}{p}+\frac{3}{q}>1\qquad\text{and}\qquad q\geq2\,.\label{eq:region-non-uniqueness}
\end{equation}
\end{thm}
We note that a Leray–Hopf solution on $(0,T)$
belongs to $L^{p}(0,T;L^{q}(\mathbb{R}^{3}))$ for all
\begin{equation}
\frac{2}{p}+\frac{3}{q}\geq\frac{3}{2}\qquad\text{and}\qquad q\in[2,6]\,.\label{eq:region-leray}
\end{equation}
by the standard Sobolev embedding.
If a Leray–Hopf solution belongs to the Serrin class $L^{p}(0,T;L^{q}(\mathbb{R}^{3}))$
with
\begin{equation}
\frac{2}{p}+\frac{3}{q}\leq1\,,\label{eq:region-serrin}
\end{equation}
then the solution is unique and smooth \citep{Prodi-uniqueness1959,Serrin-initialvalueproblem1963,Ladyzhenskaya-uniquenessandsmoothness1967,Escauriaza-L3solutions2003}.
\Thmref{localization} shows that the Serrin uniqueness criterion is essentially optimal, since non-uniqueness holds for $p$ and $q$
satisfying \eqref{region-non-uniqueness}. The Leray–Hopf and Serrin classes are represented on \figref{region}.

Our main focus in this paper is on the numerics, which are presented in \secref{methods,results}.
The proofs of \thmref{spectrum-L,continuation-bifurcation,localization} are sketched
in \secref{spectrum-LU,bifurcation,localization} respectively and,
in general, go along the lines similar to those in \cite{Jia-Areincompressible3d2015}.

\medskip\noindent{\bfseries Notations.} The spaces $\mathcal{U}$, $\mathcal{V}$, and $\mathcal{D}$
are defined by \eqref{def-U}, \eqref{def-V}, and \eqref{def-D}
respectively, and the subspaces of axi-symmetric vector fields are
denoted by $\mathcal{U}_{\axi}$, $\mathcal{V}_{\axi}$,
and $\mathcal{D}_{\axi}$ respectively.
The operators $F$ and $\cL$ are respectively defined by \eqref{def-F} and \eqref{def-L}.
The cylindrical coordinates are denoted by $(r,\theta,z)$.
If $\alpha$ is a multi-index, we denote by $\proj_{\alpha}$ the projection
on the elements of $\alpha$, $\proj_{\alpha}\bv=\sum_{i\in\alpha}\left(\bv\bcdot\be_{i}\right)\be_{i}$.
For example, $\proj_\theta\bu=u_\theta\be_\theta$ and $\proj_{rz}\bu=u_r\be_r+u_z\be_z$.

\begin{figure}[h]
\includefigure{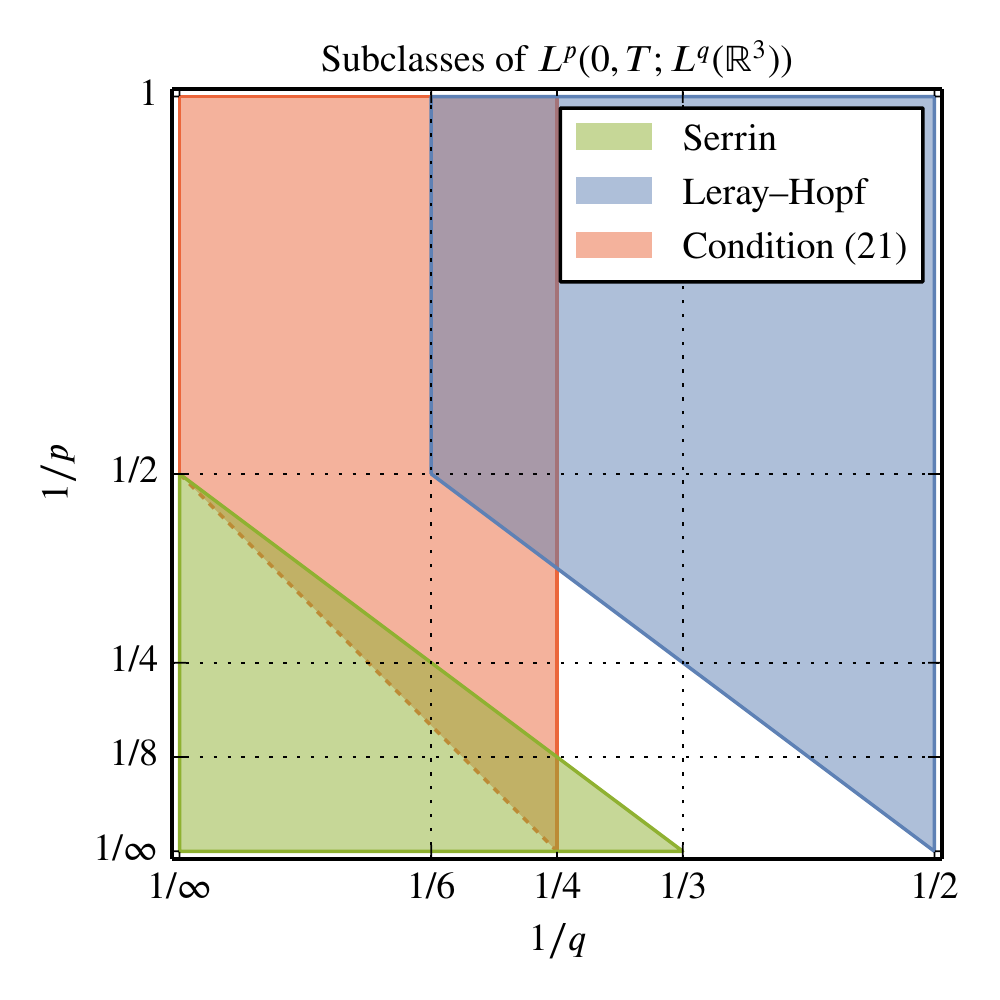}

\caption{\label{fig:region}Different subclasses of the space $L^{p}(0,T;L^{q}(\mathbb{R}^{3}))$.
The Leray–Hopf solutions belong to the blue region characterized by
\eqref{region-leray}. The Serrin criterion for uniqueness and regularity
is given by the green region defined by \eqref{region-serrin}.
The corrector used in the proof of \thmref{localization} to localize a self-similar
solution belongs to the space $X_{T}$ defined by \eqref{def-X_T},
hence to the red region characterized by \eqref{reqion-w}.
The localization of the numerical solutions found belong to the complement
of the green region, hence showing that the Serrin uniqueness criterion
is sharp in these spaces.}
\end{figure}

\section{\label{sec:methods}Numerical methods}

The restriction to the subspace of axi-symmetric solutions allows to
perform the numerical simulations in a two-dimensional domain in the
$(r,z)$ coordinates. We work in the following computational domain
\[
\Omega(R)=\bigl\{(r,z)\in\mathbb{R}^{2}:0\leq r\leq R\;\text{and}\;|z|\leq R\bigr\},
\]
and divide its boundary into two disjoint parts, $\partial\Omega(R)=A(R)\cup\Gamma(R)$,
where
\[
A(R)=\bigl\{(0,z)\in\mathbb{R}^{2}:|z|<R\bigr\}
\]
is the axis boundary and $\Gamma(R)$ the artificial boundary. As
it will become clear later, when the parameter $\sigma$ is increasing,
the domain as to be also increasing in order to keep the region of
interest into the computational domain. Here we choose to work in
the domain $\Omega(R_{\sigma})$, where $R_{\sigma}=20\scale_{\sigma}$
with $\scale_{\sigma}^{2}=1+\frac{\sigma}{4}$. This specific factor
was chosen such that visually the interesting phenomena are approximately
located in the same region of the computational domain for all values
of $\sigma$.

The cylindrical coordinates require the following boundary condition on the axis,
\[
\bUsigma\bcdot\be_{r}=0\qquad\text{and}\qquad\bUsigma\bcdot\be_{\theta}=0\qquad\text{on}\qquad A(R_{\sigma})\,.
\]
The condition \eqref{ns-scale-bc} naturally leads to the following boundary
condition on $\Gamma(R_{\sigma})$,
\[
\bUsigma=\sigma\ba_{0}\qquad\text{on}\qquad\Gamma(R_{\sigma})\,.
\]
The reader not interested in the implementation of the numerical simulations
can safely jump to \secref{results} for the description of the numerical
results.

\subsection{\label{sub:methods-discretization}Discretization}

The numerical simulations are performed by the finite elements method
with the package FEniCS \citep{Logg-FENICS2012,Alnes-FENICS2015}.

The domain $\Omega(R)$ is first discretized into $2n^{2}$ squares
each of them split into two triangles, as shown on \figref{mesh}a.
To increase the precision near the origin, this discretization is
refined in the square $r\leq R/2$ and $|z|\leq R/4$, which leads
to the discretization $\Omega(R,n)$ represented on \figref{mesh}b.
As already said, we need to work in a domain growing as $\sigma$
is increasing. In order to keep the mesh fixed during the continuation
in $\sigma$, we instead choose to rescale the equations \eqref{ns-scale}
in $\Omega(R_{\sigma})$ by a factor $\scale_{\sigma}^{2}=1+\frac{\sigma}{4}$.
That way, the the domain $\Omega(R_{\sigma})$ is transformed into
the domain $\Omega(20)$ and the same mesh can be used for all the
values of $\sigma$.

The following weak formulation of \eqref{ns-scale-eq} is used
\[
\begin{aligned}\bigl(\bnabla\bUsigma,\bnabla\bvphi\bigr)-\frac{1}{2}\bigl(\bx\bcdot\bnabla\bUsigma,\bvphi\bigr)-\frac{1}{2}\bigl(\bUsigma,\bvphi\bigr)+\bigl(\bUsigma\bcdot\bnabla\bUsigma,\bvphi\bigr)+\bigl(P_{\!\sigma},\bnabla\bcdot\bvphi\bigr) & =0\,,\\
\bigl(\bnabla\bcdot\bU_{\sigma},q\bigr) & =0\,,
\end{aligned}
\]
where $\bigl(\cdot,\cdot\bigr)$ denotes the scalar product on $L^{2}(\Omega(R_{\sigma}))$
and $\bvphi$ and $q$ are test functions. The restriction of this
weak formulation to axi-symmetric is transformed into cylindrical coordinates
and then discretized with Lagrange quadratic polynomials (P2 elements)
for $\bUsigma$ and linear polynomials (P1 elements) for $P_{\!\sigma}$.

\begin{figure}[p]
\includefigure{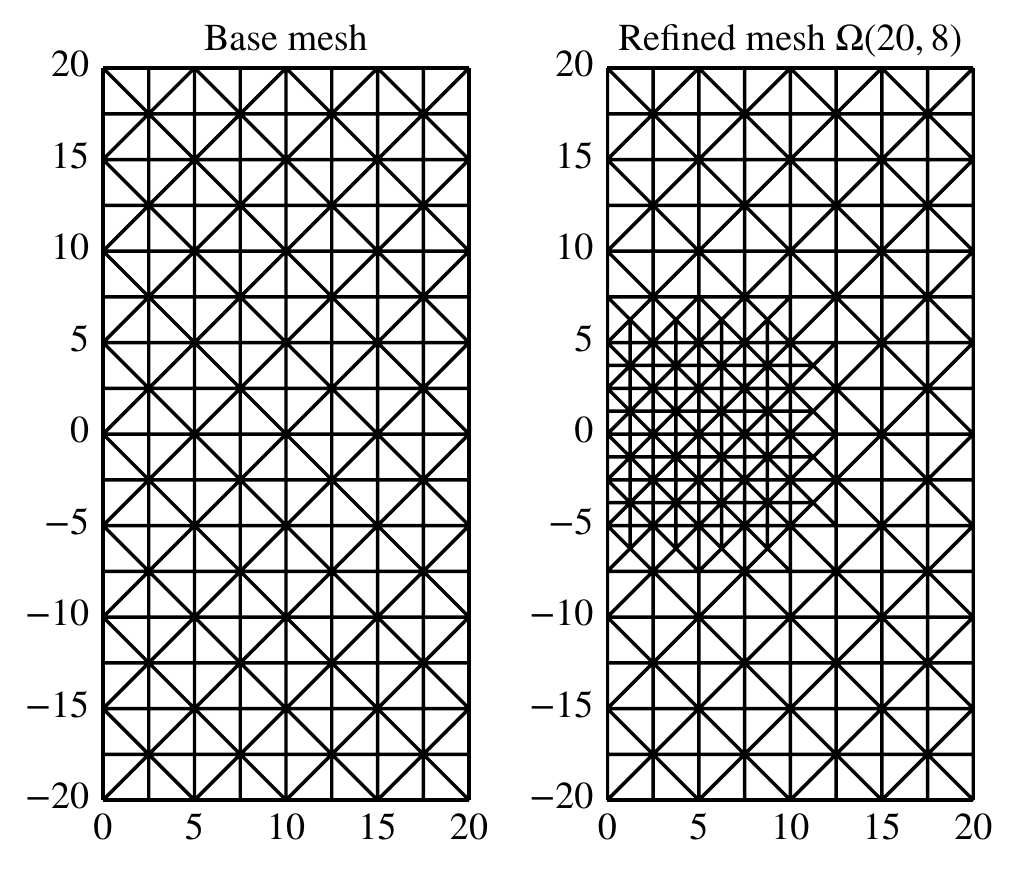}\caption{\label{fig:mesh}Construction of the discretization $\Omega(R,n)$
for $R=20$ and $n=8$. First the domain $\Omega(R)$ is discretized
into $2n^{2}$ squares and then refined near the origin. }
\end{figure}

\subsection{\label{sub:methods-continuation}Continuation algorithm for $\bUsigma$}

In a first step, a continuation method is used in $\sigma$ on the
domain $\Omega(R_{\sigma},300)$. The steps of the continuation method
are chosen as $0.1$ for $0\leq\sigma\leq2$, $0.5$ for $2\leq\sigma\leq50$
and to $1$ for $50\leq\sigma\leq500$. At each step the solution
from the previous step is used as an initial datum for a Newton's method.
This Newton's method typically converges in two or three steps. This
method was used because adjusting the step such that only one Newton's
iteration leads to convergence is much too slow. In a second step,
the solution $\bUsigma$ founded on $\Omega(R_{\sigma},300)$ is interpolated
into the finer mesh $\Omega(R_{\sigma},600)$. From this initial guess,
only one Newton's iteration leads to a converged solution on $\Omega(R_{\sigma},600)$
in general. All the Newton's iterations are performed with the MUMPS
\citep{Amestoy-MUMPS2000} linear solver through PETSc \citep{Balay-PETSC2016}
binding.

\subsection{\label{sub:methods-eigensolver}Eigenvalues solver}

On $\Omega(R_{\sigma})$, the eigenvalue problem of $\cL(\bUsigma)$
is given by
\begin{align*}
-\Delta\bphi-\frac{\bx}{2}\bcdot\bnabla\bphi-\frac{1}{2}\bphi+\bUsigma\bcdot\bnabla\bphi+\bphi\bcdot\bnabla\bUsigma+\bnabla p & =\lambda\bphi\,, & \bnabla\bcdot\bphi & =0\,,
\end{align*}
with the boundary conditions
\[
\bphi\bcdot\be_{r}=0\qquad\text{and}\qquad\bphi\bcdot\be_{\theta}=0\qquad\text{on}\qquad A(R_{\sigma})\,,
\]
and
\[
\bphi=\bzero\qquad\text{on}\qquad\Gamma(R_{\sigma})\,.
\]
These equations are solved in the class of axi-symmetric $\bphi$ and
the discretization used is $\Omega(R_{\sigma},600)$ in the same way
as explained in \subref{methods-discretization}. Due to the truncation
of the domain, only the eigenvectors $\bv$ of $\cL(\bUsigma)$
in $\mathbb{R}^{3}$ having a relatively fast decay at infinity will be
found. For a local equation in a similar situation it might be reasonable to expect that eigenvectors with exponential decay exist. However, due to non-local effect in the Navier–Stokes equations, the fastest decay one can expect in our problem here is probably $O(|\bx|^{-4})$, as the terms generated by the original non-linearity need to be projected on divergence-free fields, which creates long-range terms.
Therefore, imposing a Dirichlet boundary conditions on the eigenvectors deforms the problem slightly. In practical calculations this effect did not seem to be significant. For a computer-assisted proof this issue would of course have to be carefully addressed. One possibility for this would be to work with the asymptotic expansions at the spatial infinity, as we already discussed above.

In a first step the 36 eigenvalues closest to the real axis were computed
for each values of $\sigma$ by using the Krylov–Schur algorithm \citep{Hernandez-KrylovSchur2009}
implemented in SLEPc \citep{Hernandez-SLEPc2005}. Instead of choosing
a random initial vector, a linear combination of the eigenvectors
founded at the previous step is used, even if the gain in the execution
time is not very large.

In a second step, we track the eigenvalues closest to the real axis
by a continuation method back to $\sigma=0$ in order to assert that
they are not spurious and actually linked to the eigenvalues at $\sigma=0$.
For this continuation by used the Newton's method \citep{Rall-Newton1961,Anselone-Newton1968}
by viewing the eigenvalue problem as a non-linear one with a constraint
on the size of the eigenvector.

\subsection{\label{sub:methods-bifurcation}Bifurcation from the $\mathcal{R}$-symmetric
solution}

In the scenario where a real eigenvalue is crossing the real axis
at $\sigma_{0}$, as supposed in the hypotheses of \thmref{continuation-bifurcation},
then another solution of \eqref{ns-scale} should bifurcate from $\bUsigma$
at $\sigma=\sigma_{0}$. This new branch of solution can be also found
numerically. For a value of $\sigma$ slightly bigger than $\sigma_{0}$,
Newton's iterations are performed with the initial guess $\bUsigma+\alpha\bphi$,
where $\bphi$ is the eigenvector corresponding to the crossing eigenvalue
and $\alpha\in\mathbb{R}$ is some real parameter to be adjusted such
that the Newton's method converges. When $\alpha$ is well chosen,
the Newton's method converges to a solution $\bUsigma+\bVsigma$ different
from $\bUsigma$. Finally the continuation algorithm described in
\subref{methods-continuation} is used to determine the new branch
of solution $\bUsigma+\bVsigma$ for larger values of $\sigma$.

\section{\label{sec:results}Numerical results}

\subsection{Base solution $\bUsigma$}

Using the continuation algorithm described in \subref{methods-continuation},
an axi-symmetric and $\mathcal{R}$-symmetric solution $\bUsigma$
was found for $\sigma\in[0,500]$. This solution is represented on
the whole computational domain $\Omega(R_{\sigma},600)$ in \figref{rut,rurz}.
Near the vertical axis, the radial and azimuthal components of $\bUsigma$
behaves like $O(r)$ for small values of $r$ has required by the
smoothness of the solution. The solutions are $(-1)$-homogeneous on
a quite large region near the artificial boundary $\Gamma(R_{\sigma})$
as shown on \figref{decay_u,decay_gradu}. This means that the choice
of the size of the computational domain $\Omega(R_{\sigma})$ was
large enough. Near the origin, the solution is shown on \figref{urz_small,ut_small,rut_small}.
As shown on \figref{urz_small}, the streamlines projected on the
plane $\theta=0$ are closed, therefore the streamlines of the profile
$\bUsigma$ are given by tori as shown on \figref{stream}. The first numerical observation \vpageref{enu:steady} concerning the existence of $\bUsigma$ is shown.

\begin{figure}[p]
\includefigure{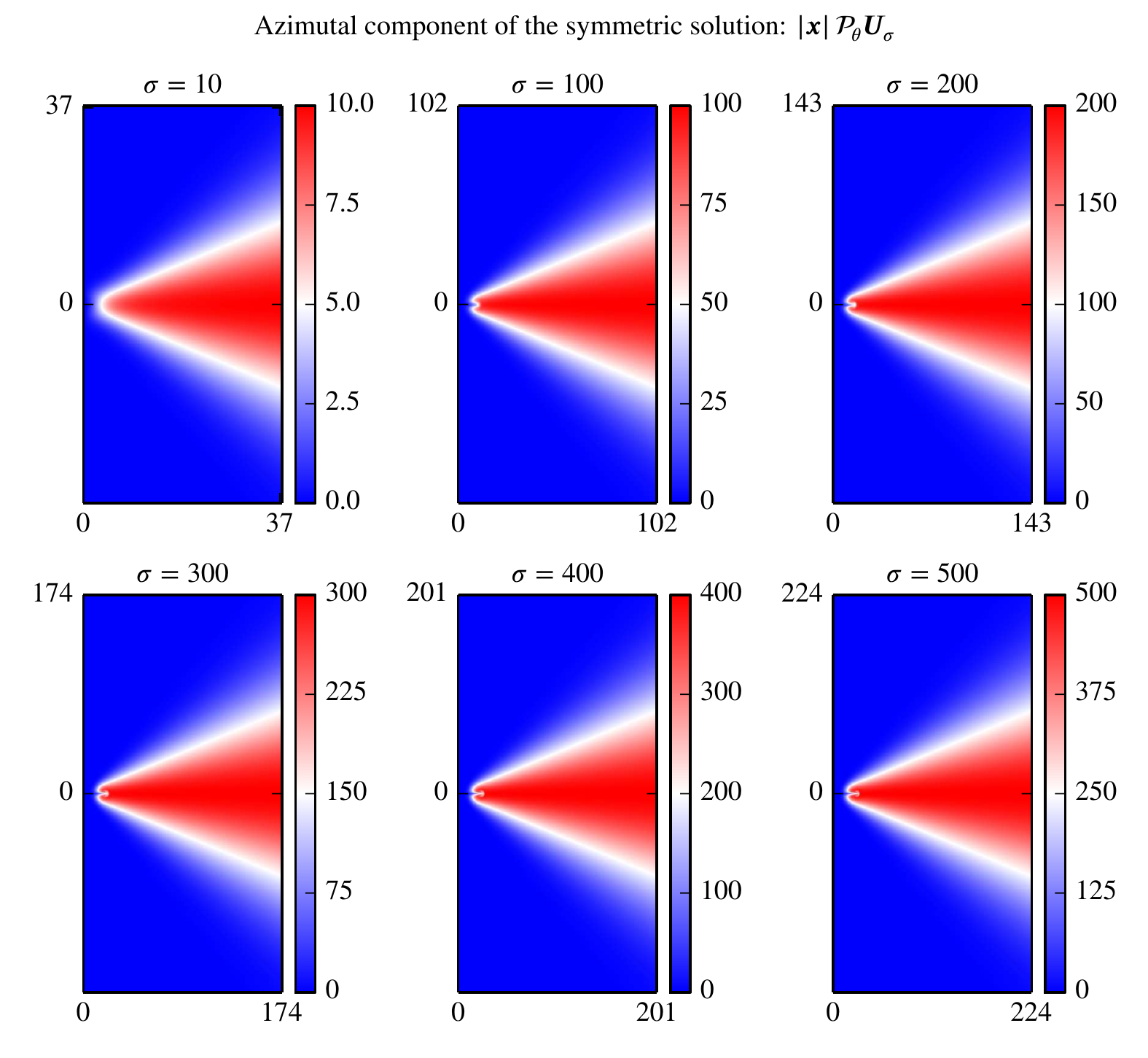}\caption{\label{fig:rut}Azimuthal component of the numerical solution $\bUsigma$
multiplied by $|\bx|$ on the whole computational domain $\Omega(R_{\sigma},600)$
for various $\sigma$. One can see that the choice of $\scale_{\sigma}$
is made such that the solution remains $(-1)$-homogeneous in most
of the computational domain except near the origin.}
\end{figure}

\begin{figure}[p]
\includefigure{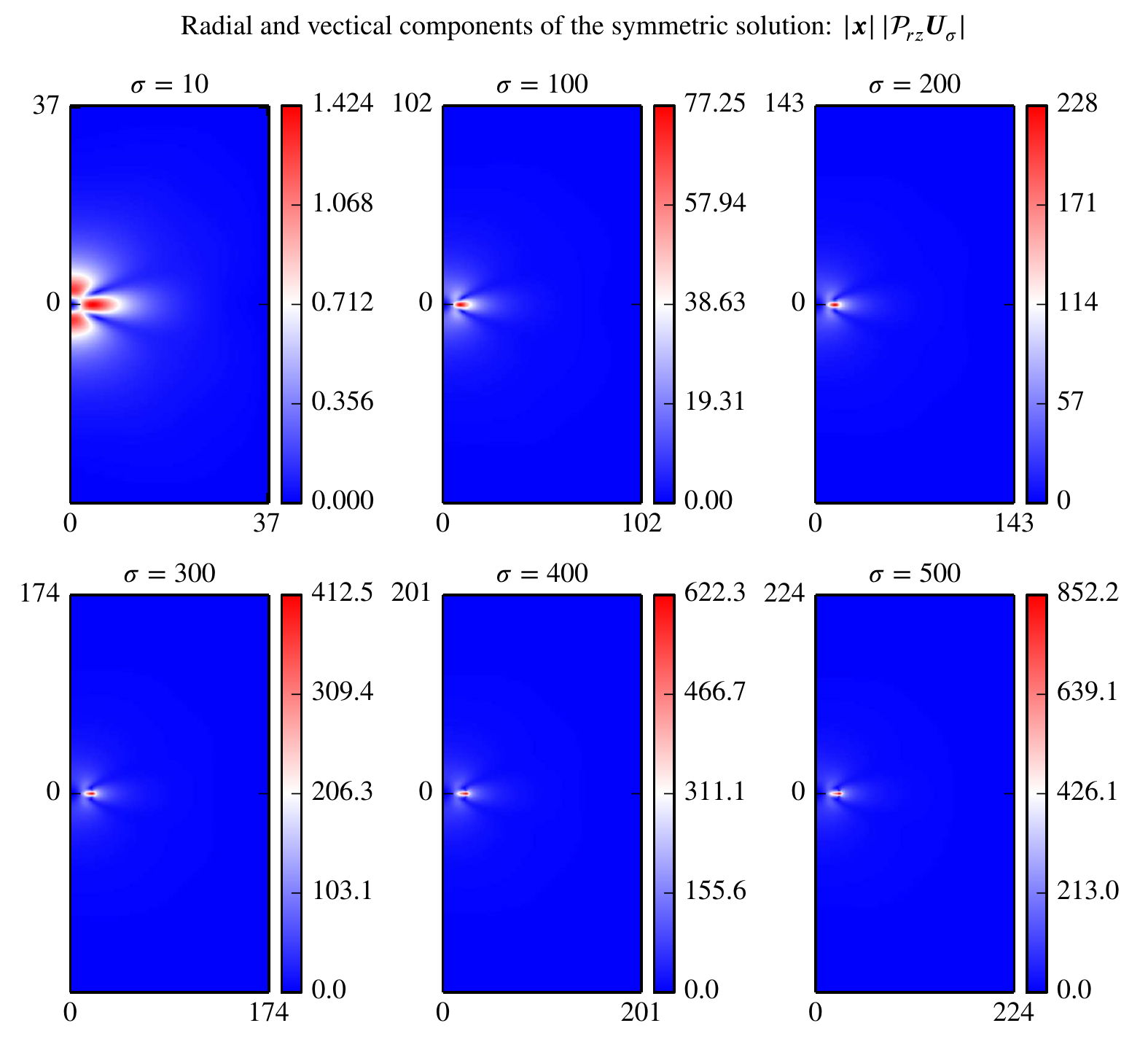}\caption{\label{fig:rurz}Norm of the radial and vertical component of the
numerical solution $\bUsigma$ multiplied by $|\bx|$.
As expect since the boundary condition $\sigma\ba_{0}$ is
pure swirl, these two components decays like $|\bx|^{-3}$.}
\end{figure}

\begin{figure}[p]
\includefigure{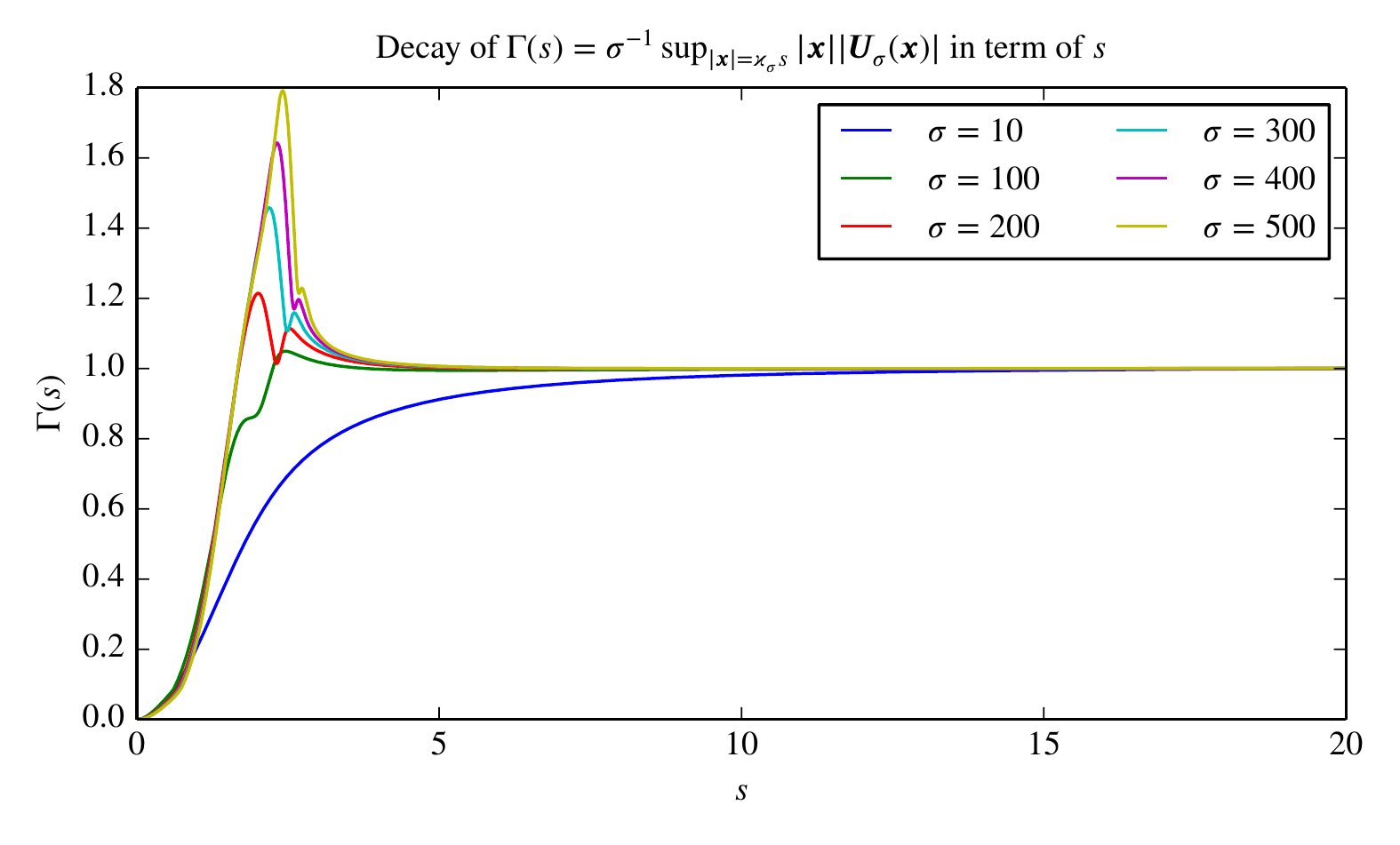}\caption{\label{fig:decay_u}Decay of the function $\Gamma(s)=\sigma^{-1}\sup_{|\bx|=\scale_{\sigma}s}|\bx||\bUsigma(\bx)|$
in term of $s\in(0,20)$ for various values of $\sigma$. The function
$\Gamma(s)$ is almost flat for $s\geq10$, so this means that the
computational domain is large enough, since the numerical solution
is already $(-1)$-homogeneous in a large region.}
\end{figure}

\begin{figure}[p]
\includefigure{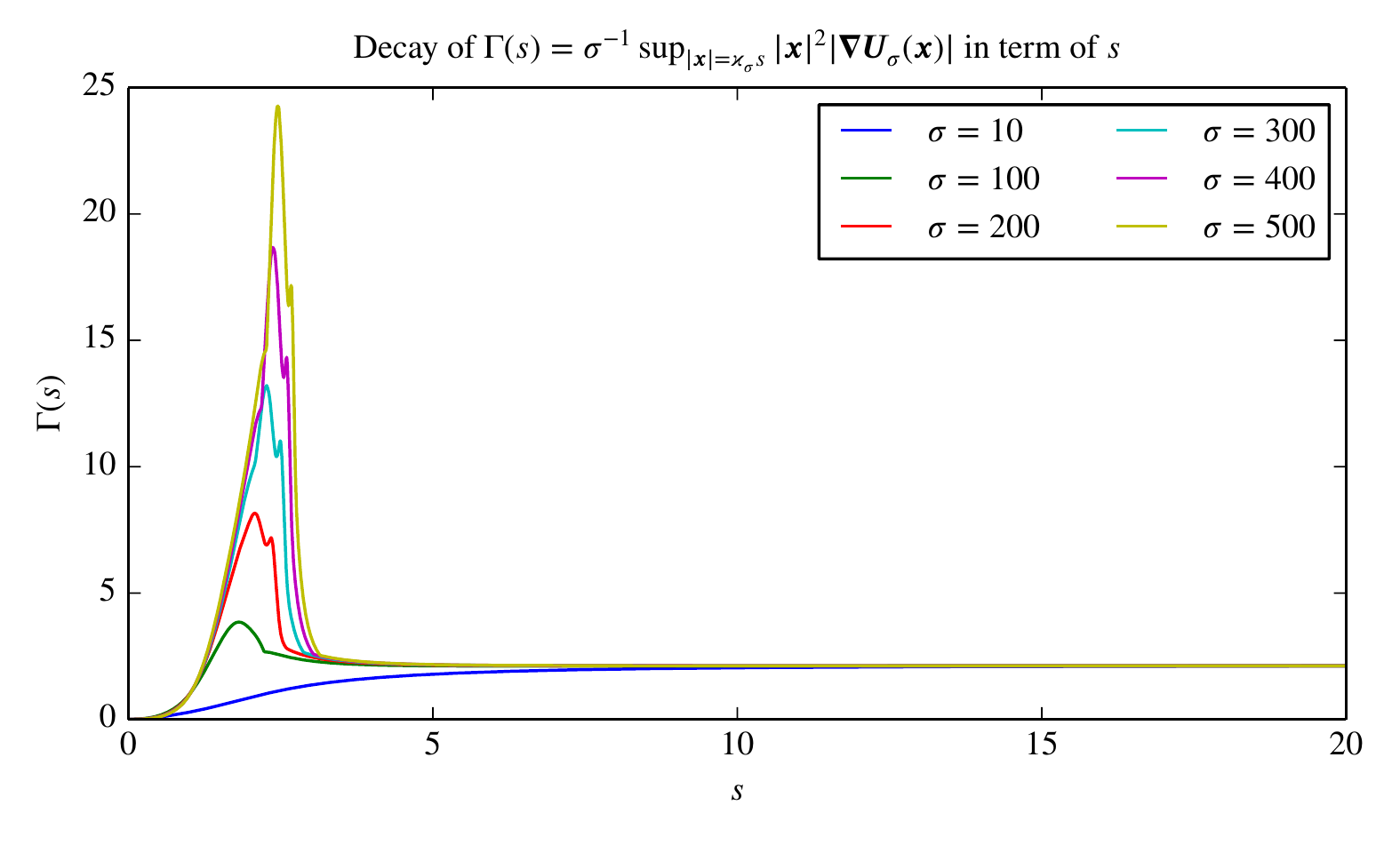}\caption{\label{fig:decay_gradu}Decay of the function $\Gamma(s)=\sigma^{-1}\sup_{|\bx|=\scale_{\sigma}s}|\bx|^{2}|\bnabla\bUsigma(\bx)|$
in term of $s\in(0,20)$ for various values of $\sigma$. We see that
$\bnabla\bUsigma$ is already $(-2)$-homogeneous on
almost half of the computational domain.}
\end{figure}

\begin{figure}[p]
\includefigure{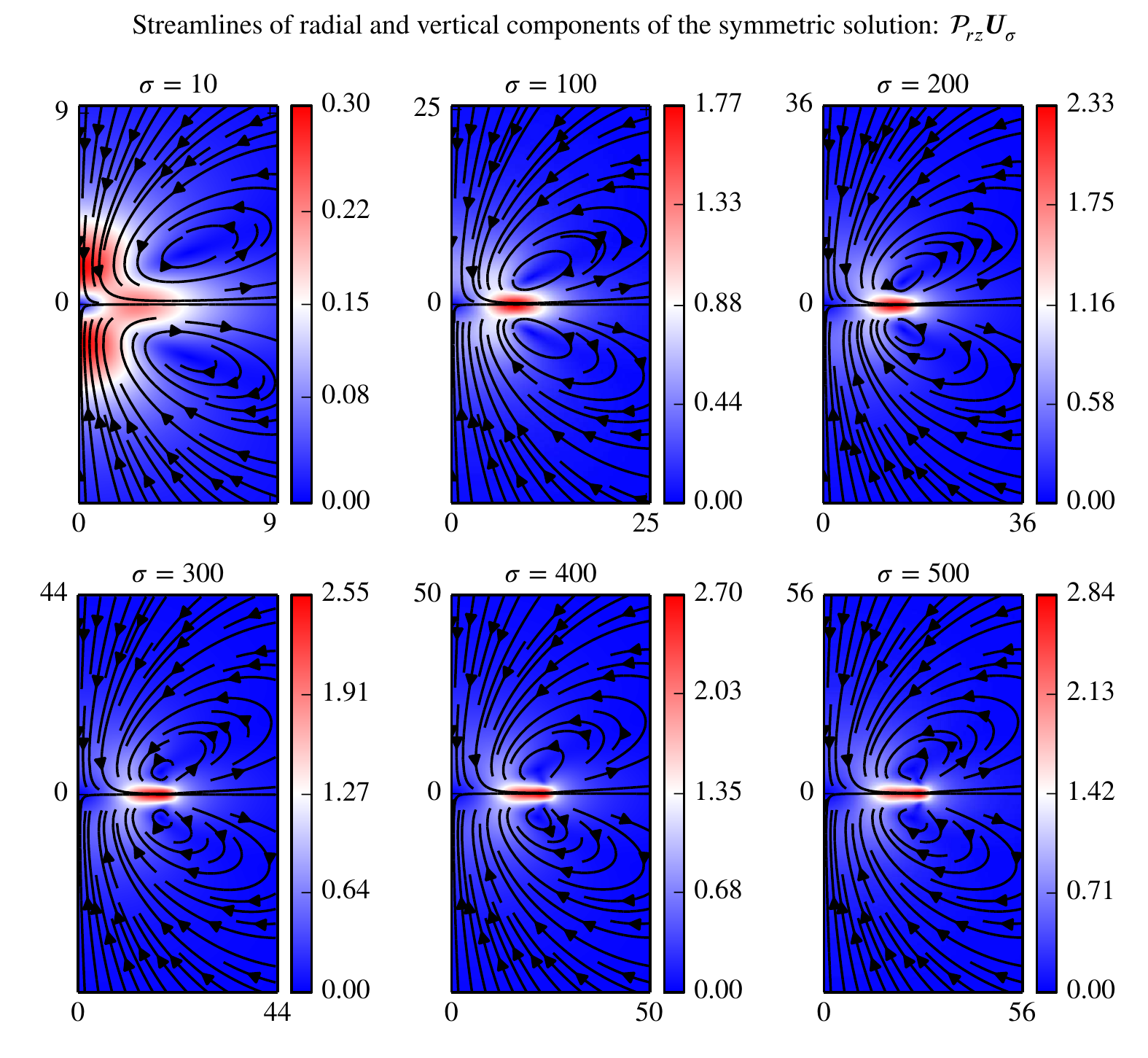}\caption{\label{fig:urz_small}Streamlines of the radial and vertical components
of $\bUsigma$ near the origin. We remark that the streamlines are closed.}
\end{figure}

\begin{figure}[p]
\includefigure{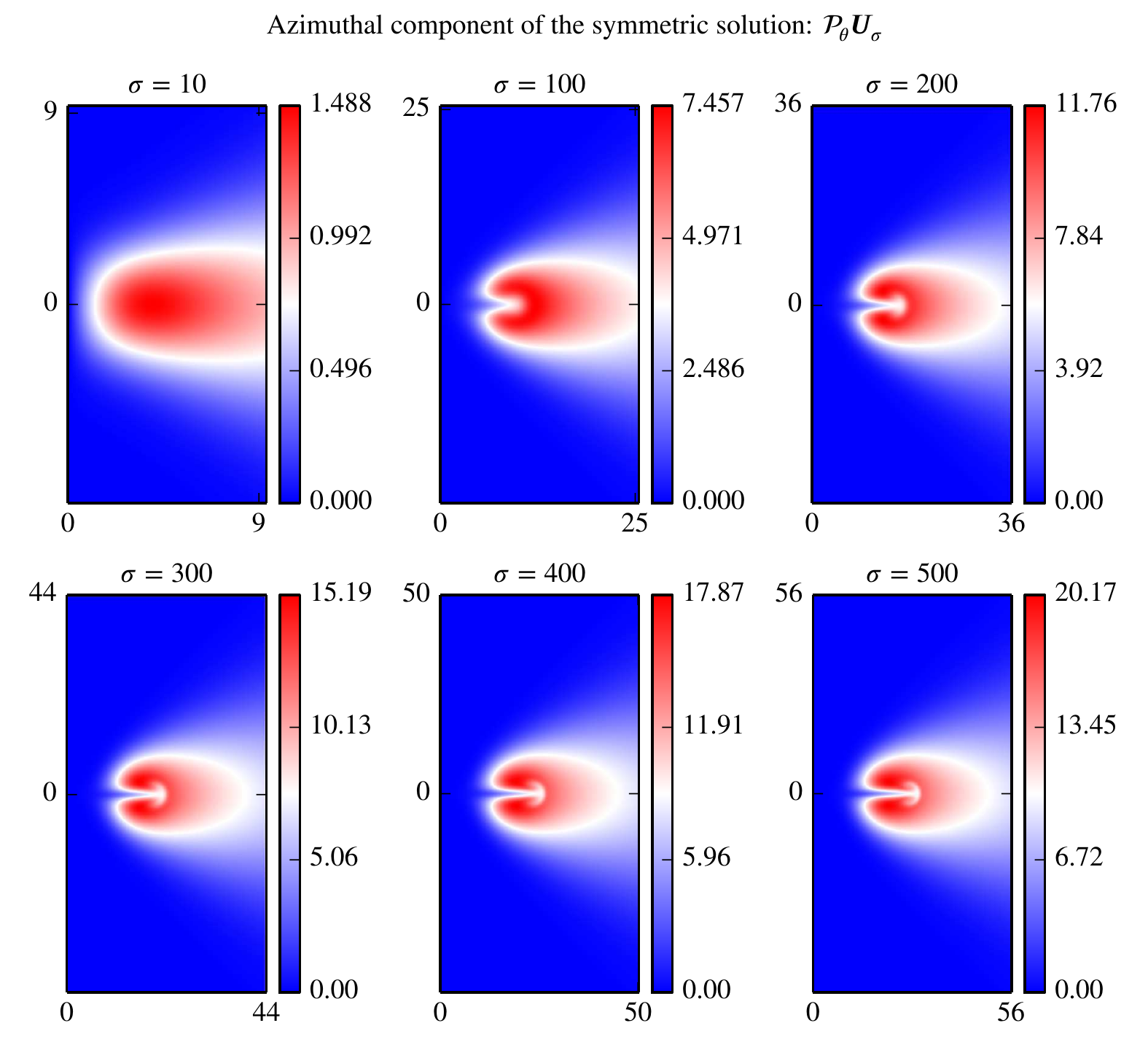}\caption{\label{fig:ut_small}Azimuthal component of the numerical solution
$\bUsigma$ near the origin.}
\end{figure}

\begin{figure}[p]
\includefigure{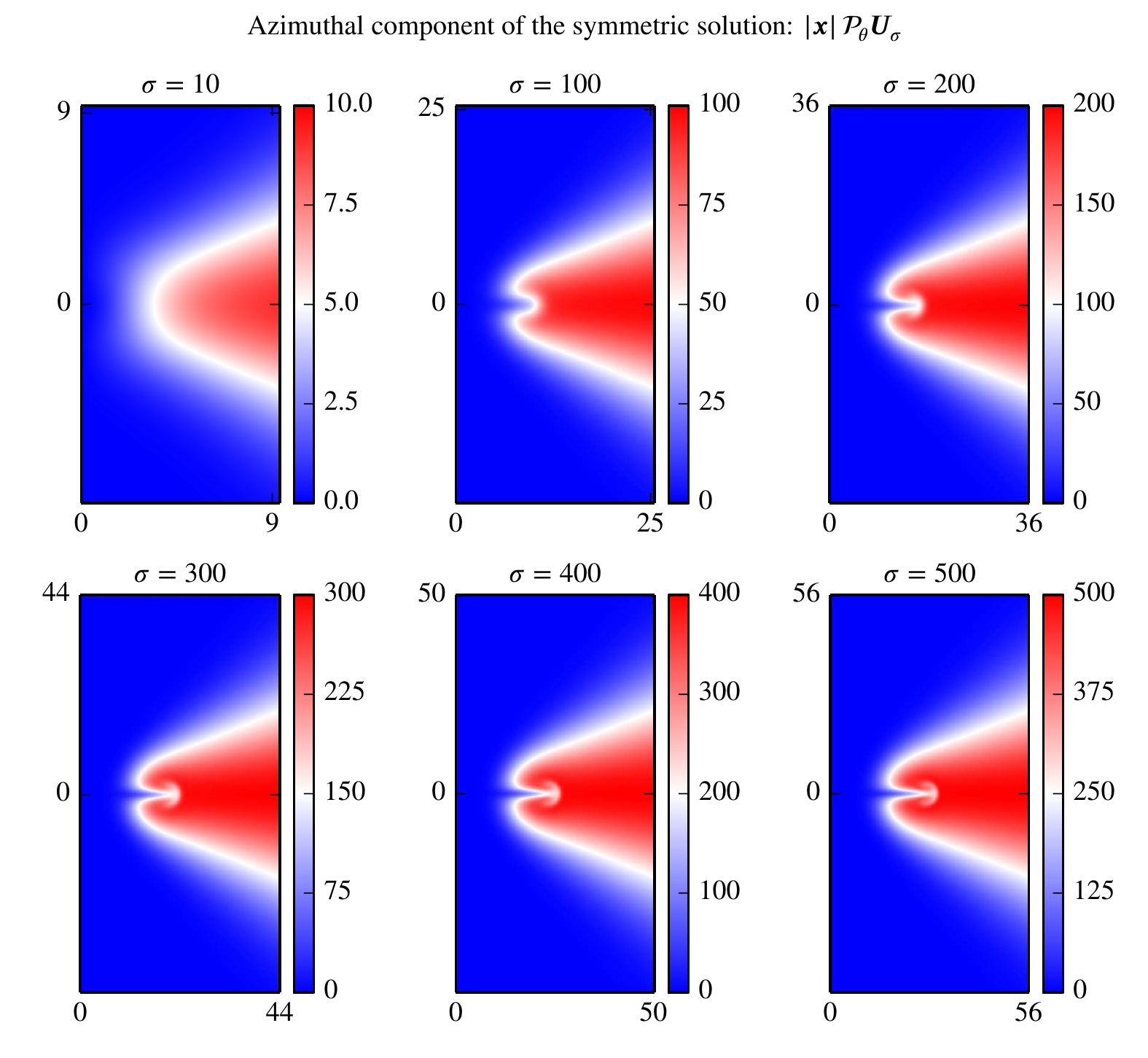}\caption{\label{fig:rut_small}Azimuthal component near the origin of the numerical
solution $\bUsigma$ multiplied by $|\bx|$.}
\end{figure}

\begin{figure}[p]
\includefigure{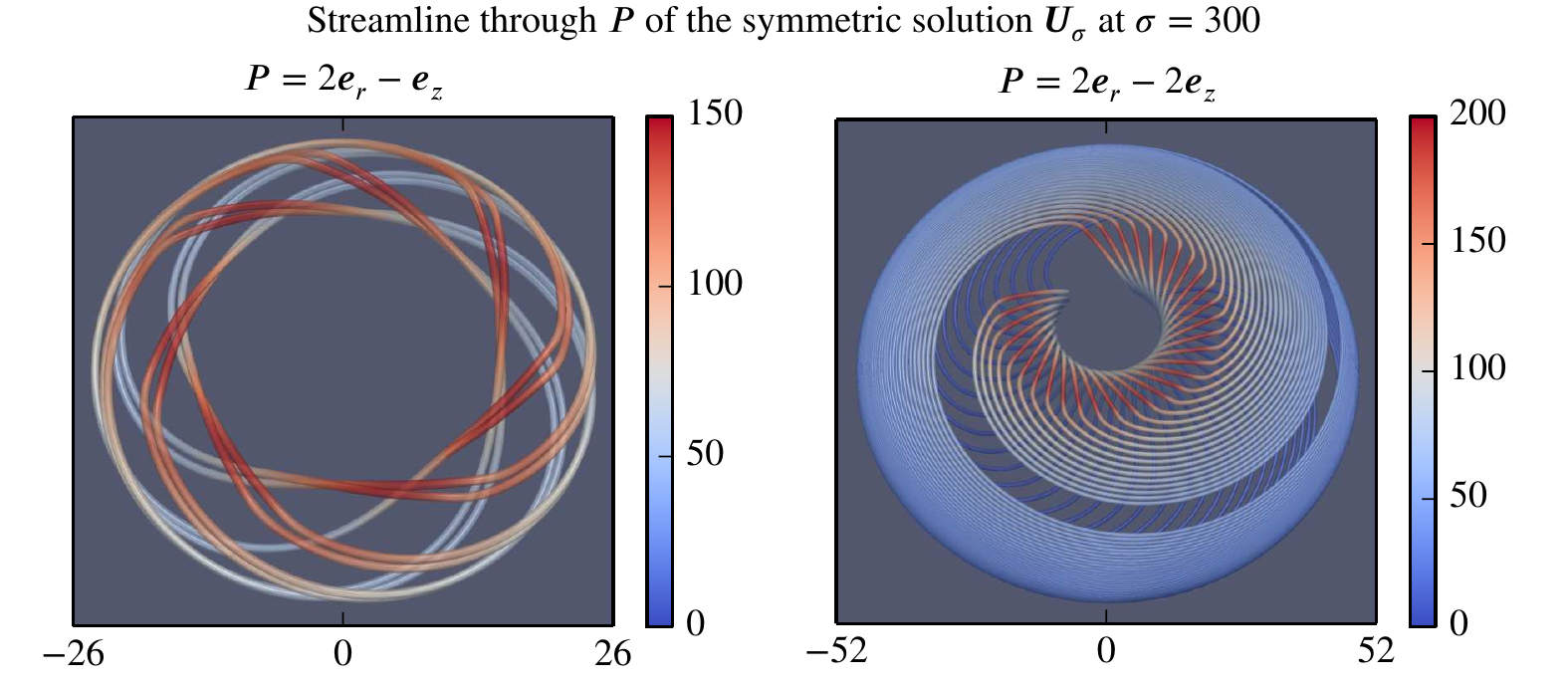}\caption{\label{fig:stream}Streamline of $\bUsigma$ at $\sigma=300$
going through two different points. (a) the point is given by $r=2$
and $z=-1$, and in this case $\proj_{\theta}\bUsigma$ is quite
large, so the streamline is almost $\frac{2\pi}{7}$-periodic in $\theta$;
(b) the point is given by $r=2$ and $z=-2$, and $\proj_{\theta}\bUsigma$
is much smaller, the streamline is almost $2\pi$-periodic in $\theta$.}
\end{figure}

\subsection{Eigenvalues of $\cL(\bUsigma)$}

At $\sigma=0$, the eigenvalues found numerically are given up to
a very high precision by $\frac{3}{2}+\frac{n}{2}$ for $n\in\mathbb{N}$,
with multiplicity $n+1$ and correspond exactly to the discrete part
found in \thmref{spectrum-L} decaying like $\e^{-\left|\bx\right|/4}$.
The continuous part is not seen numerically due to the polynomial
decay $\left|\bx\right|^{-2\lambda}$ of the eigenvectors. The eigenvectors
found for $\sigma>0$ are also extremely well-localized, even if numerically
the rate cannot be precisely determined due to precision issues. The real part
of the eigenvalues closer to the real axis are represented on \figref{spectrum,spectrum_small}.
In particular a real eigenvalue crosses the real axis near $\sigma\approx292$
whereas all the other eigenvalues have a strictly positive real part
on the range $\sigma\in[0,500]$. By going back in $\sigma$, the
crossing eigenvalue merges with another real eigenvalues near $\sigma\approx12$
to form a pair of complex conjugate eigenvalues having a real part
close to two. The eigenvalues near the crossing are represented on
\figref{spectrum_crossing} whereas the eigenvector corresponding
to $\sigma=292$ is shown on \figref{eigen_a292} and is not $\mathcal{R}$-symmetric. Hence, the second numerical observation claimed \vpageref{enu:spectrum} is shown.

Interestingly, the eigenvalue $\lambda=1.5$ is unchanged with respect to $\sigma$. The explanation of this fact comes from the following
simple observation, for which we are indebted to a valuable discussion with Hao Jia. The equation \eqref{ns-scale-eq} for the profile
$\bUsigma$ leads to the following equation for its momentum $\bTsigma=\bx\bwedge\bUsigma$,
\[
-\Delta\bTsigma+2\bOmegasigma-\frac{\bx}{2}\bcdot\bnabla\bTsigma+\bUsigma\bcdot\bnabla\bTsigma+\bx\bwedge\bnabla P_{\sigma}=\bzero\,,
\]
where $\bOmegasigma=\bnabla\bwedge\bUsigma$ is the vorticity of $\bUsigma$.
Therefore, the eigenvalue problem $\cL(\bUsigma)\bphi=\lambda\bphi$
can be transformed into the following equation for the momentum $\btau=\bx\bwedge\bphi$,
\[
-\Delta\btau+2\bomega-\frac{\bx}{2}\bcdot\bnabla\btau+\bUsigma\bcdot\bnabla\btau+\bv\bcdot\bnabla\boldsymbol{T}_{\sigma}+\bx\bwedge\bnabla p=\lambda\btau\,,
\]
where $\bomega=\bnabla\bwedge\bphi$ is the vorticity of $\bphi$.
By integrating this last equation over $\mathbb{R}^{3}$, we obtain
the following relation after some integrations by parts,
\[
\frac{3}{2}\int_{\mathbb{R}^{3}}\btau=\lambda\int_{\mathbb{R}^{3}}\btau\,,
\]
which explains why the eigenvalue $\lambda=1.5$ is unchanged even
for large values of $\sigma$.

\begin{figure}[p]
\includefigure{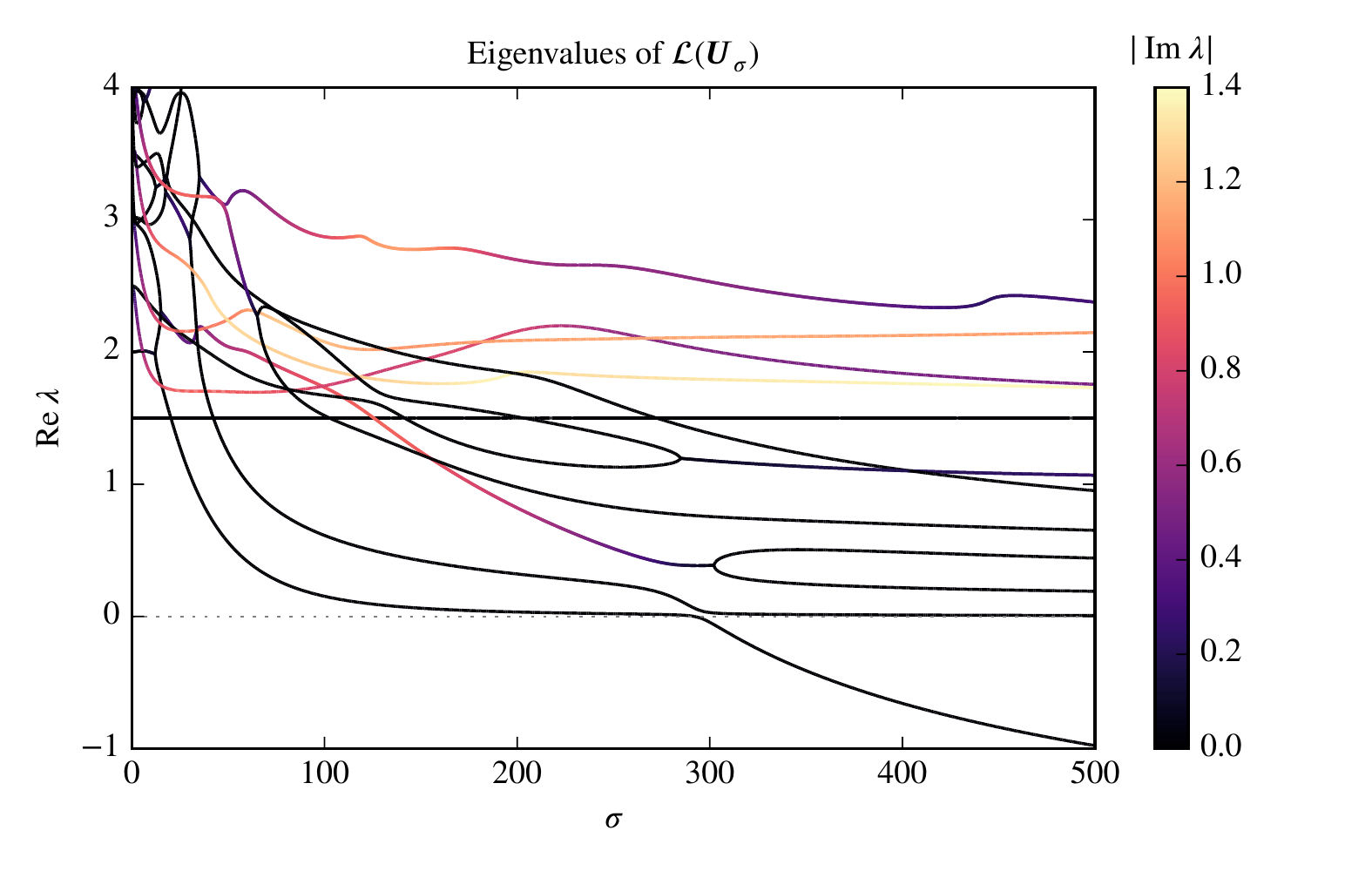}\caption{\label{fig:spectrum}Eigenvalues of $\cL(\bUsigma)$
for $\sigma\in[0,500]$. The color of the lines represents the absolute
value of the imaginary part of the eigenvalues. In order to keep the
plot readable, the calculated eigenvalues are not all represented,
but only the ones closest to the real axis.}
\end{figure}

\begin{figure}[p]
\includefigure{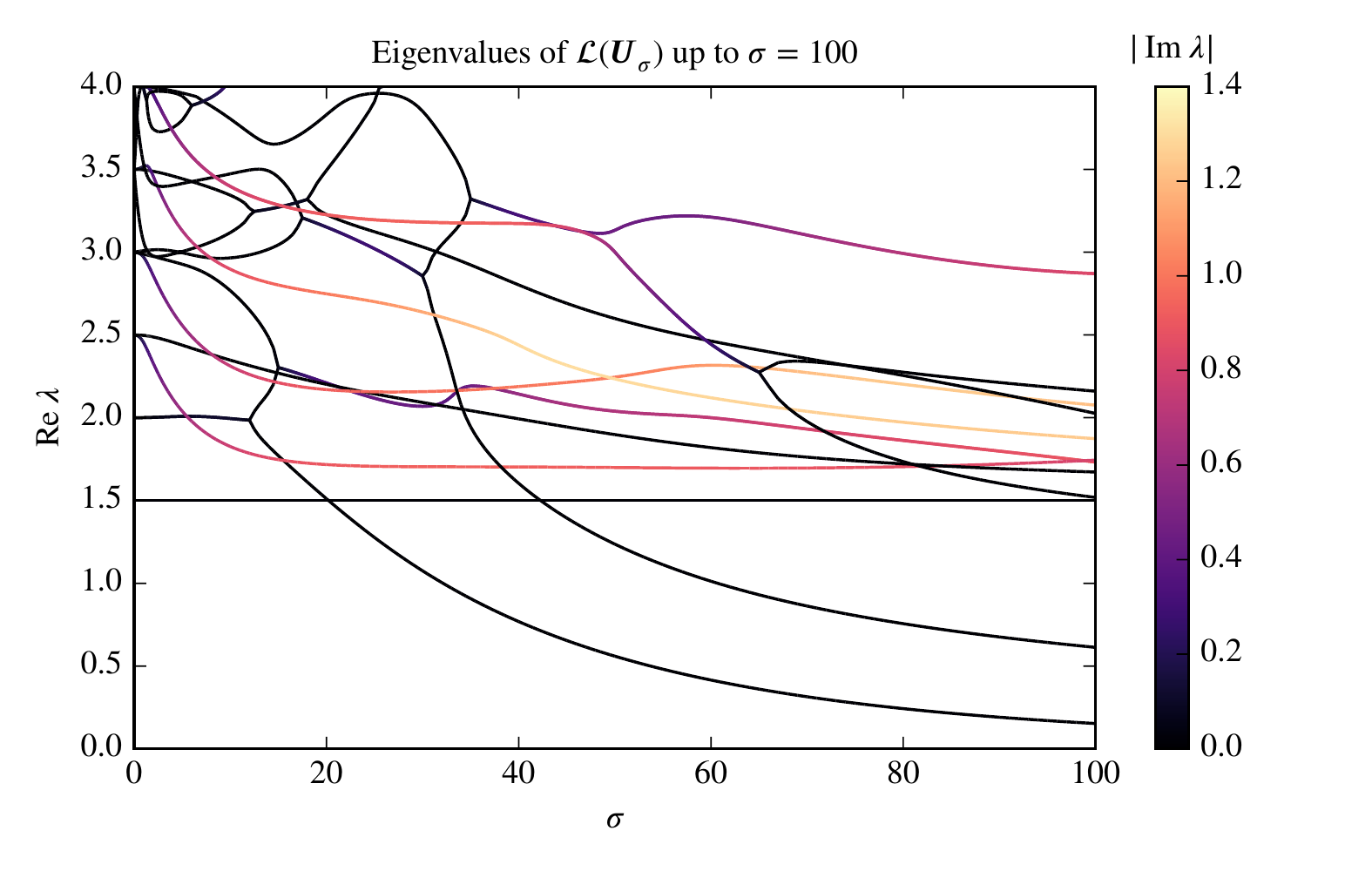}\caption{\label{fig:spectrum_small}Zoom of the eigenvalues of $\cL(\bUsigma)$
for $\sigma\in[0,100]$. At $\sigma=0$, the eigenvalues are given
by $\lambda=\frac{3}{2}+\frac{n}{2}$ for $n\in\mathbb{N}$ and their multiplicity
is $n+1$. Many different bifurcations occurs in $\sigma$.}
\end{figure}

\begin{figure}[p]
\includefigure{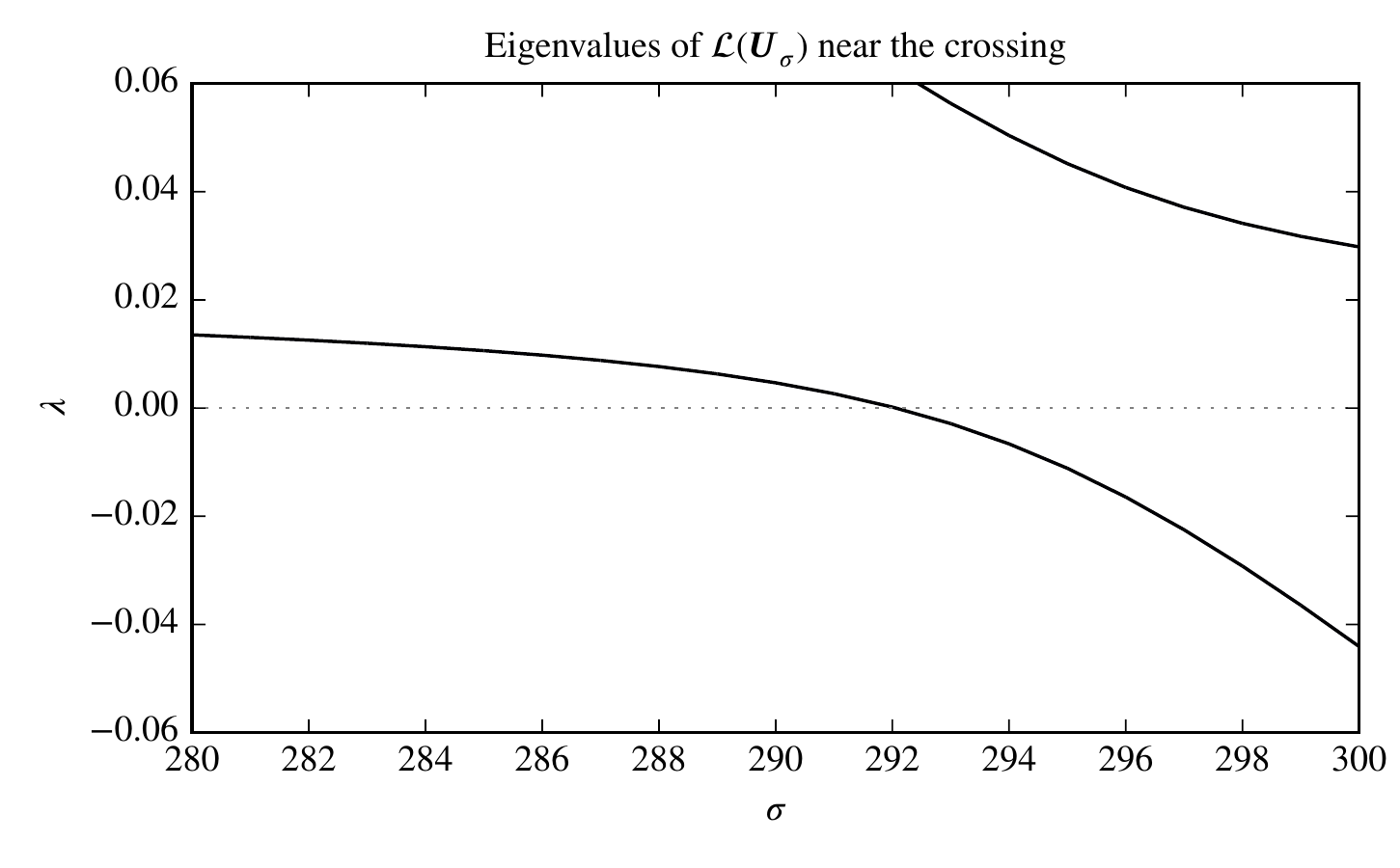}\caption{\label{fig:spectrum_crossing}Eigenvalues of $\cL(\bUsigma)$
near the crossing point $\lambda=0$. On this plot the two eigenvalues
are real.}
\end{figure}

\begin{figure}[p]
\includefigure{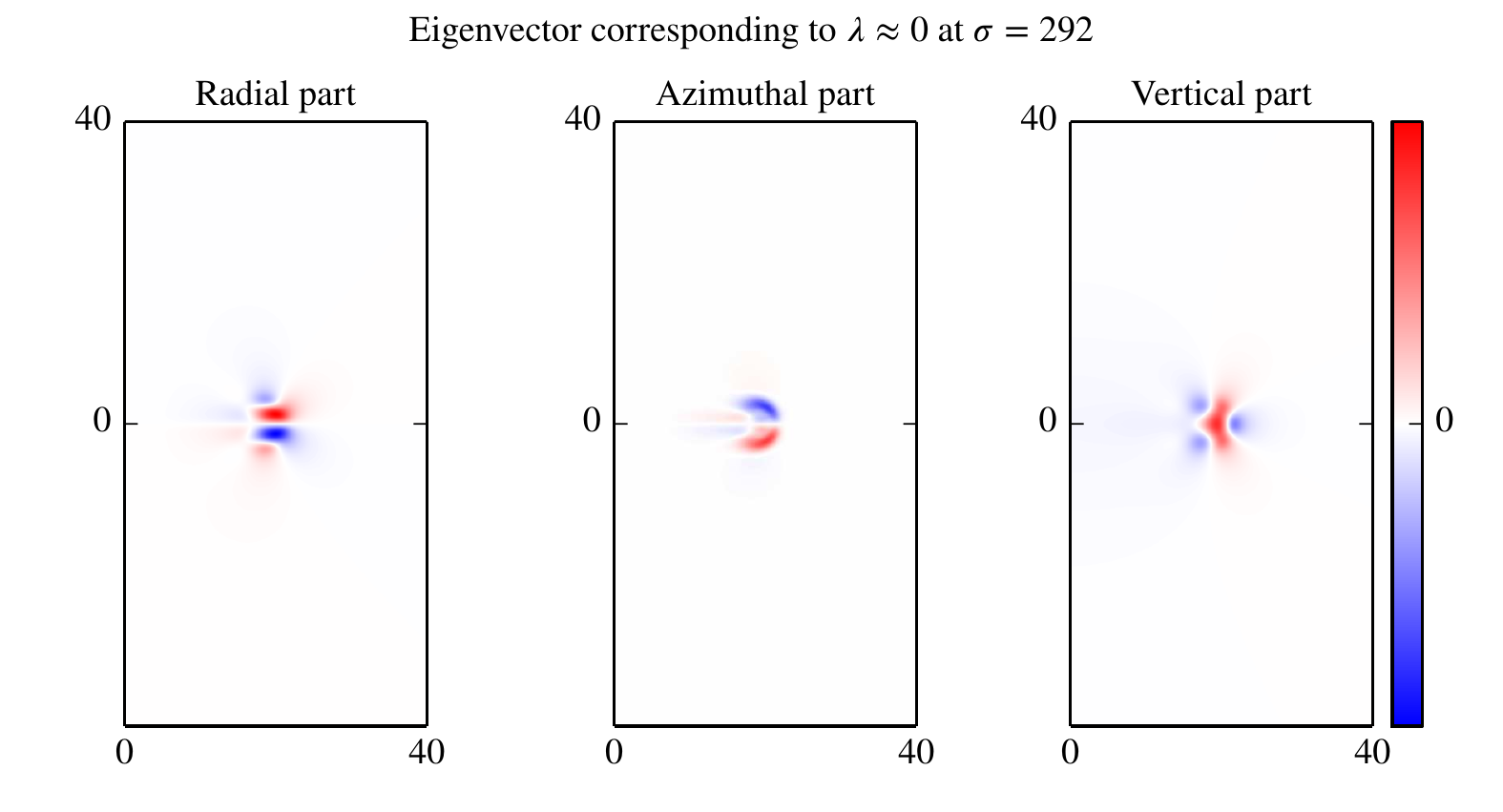}\caption{\label{fig:eigen_a292}Eigenvectors corresponding to the crossing
eigenvalue $\lambda\approx0$ at $\sigma=292$.}
\end{figure}

\subsection{Bifurcating solution $\bUsigma+\bVsigma$}

Since a real eigenvalue crossed the real axis near $\sigma\approx292$,
the method described in \subref{methods-bifurcation} furnish another
solution $\bUsigma+\bVsigma$ of \eqref{ns-scale} bifurcating from
$\bUsigma$. The bifurcating solution $\bUsigma+\bVsigma$ is no more
visually symmetric with respect to the plane $z=0$ for $\sigma\gtrsim300$
as shown on \figref{bif_urz,bif_ut,bif_rut}. More precisely, $\bVsigma=\bzero$
for $\sigma\lesssim292$ as expected and $\bVsigma$ is growing as
$\sigma$ increases for $\sigma\gtrsim292$ as shown on \figref{bif_diff}.
The reflected solution $\bUsigma+\mathcal{R}\bV_{\sigma}$ by the
plane $z=0$ is also a solution, so $\sigma\approx292$ is a supercritical
pitchfork-type bifurcation corresponding to the breaking of the
$\bZ_{2}$-symmetry with respect to the plane $z=0$. This behavior shows the
third numerical observation made \vpageref{enu:bifurcation}. By comparing the
streamlines of the base solution $\bUsigma$ (\figref{stream}a) and
of the bifurcating branches $\bUsigma+\bVsigma$ and $\bUsigma+\mathcal{R}\bV_{\sigma}$
(\figref{stream_bif}) at $\sigma=300$, we see that the topological
nature of the streamlines are drastically changed even just after
the bifurcation. The reason is that a slight change in the azimuthal
component has a large influence on the quasi-periodicity of the streamlines
on the tori.

\begin{figure}[p]
\includefigure{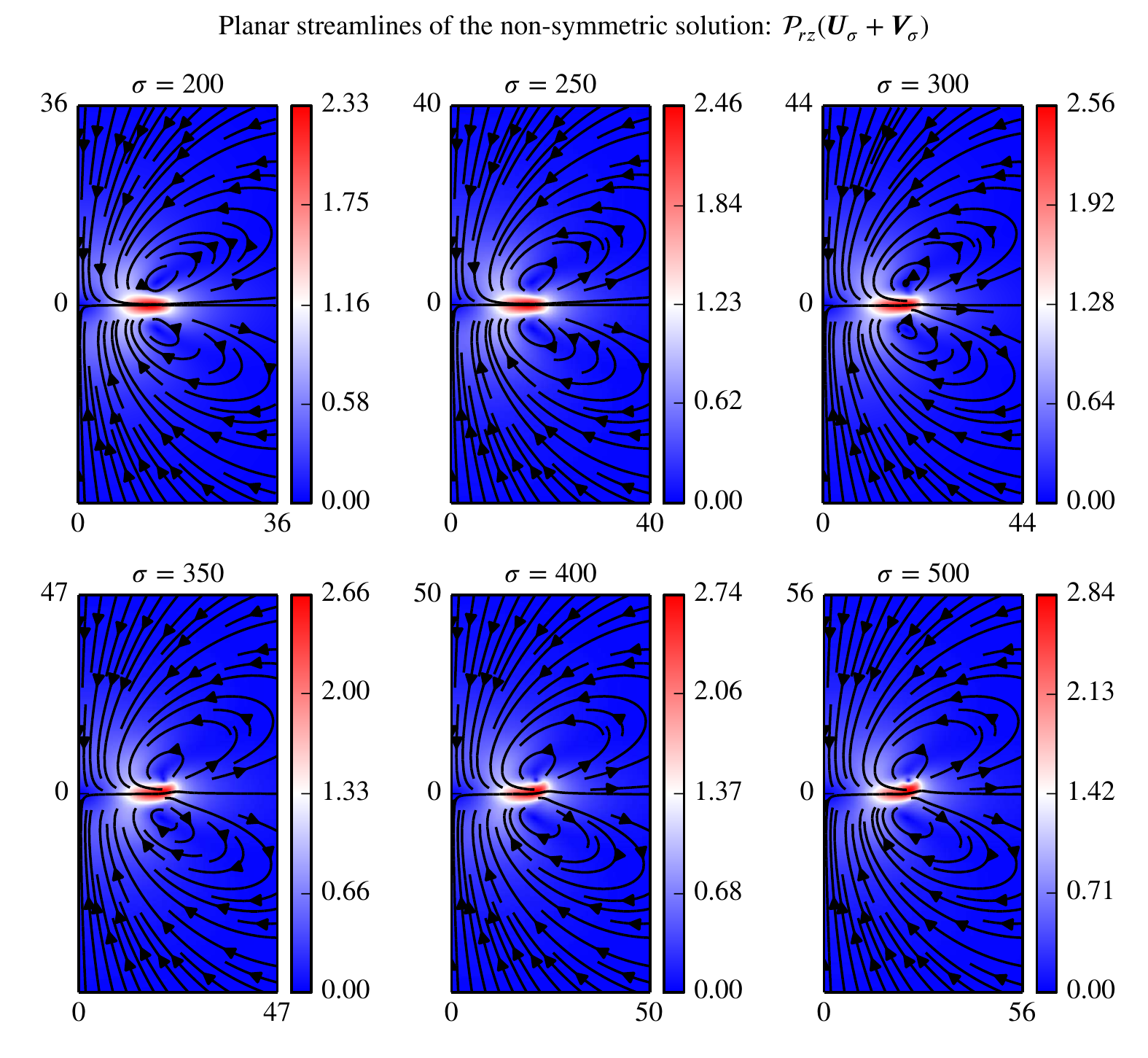}\caption{\label{fig:bif_urz}Streamlines of the radial and vertical components
of the numerical solution $\bUsigma+\bVsigma$.}
\end{figure}

\begin{figure}[p]
\includefigure{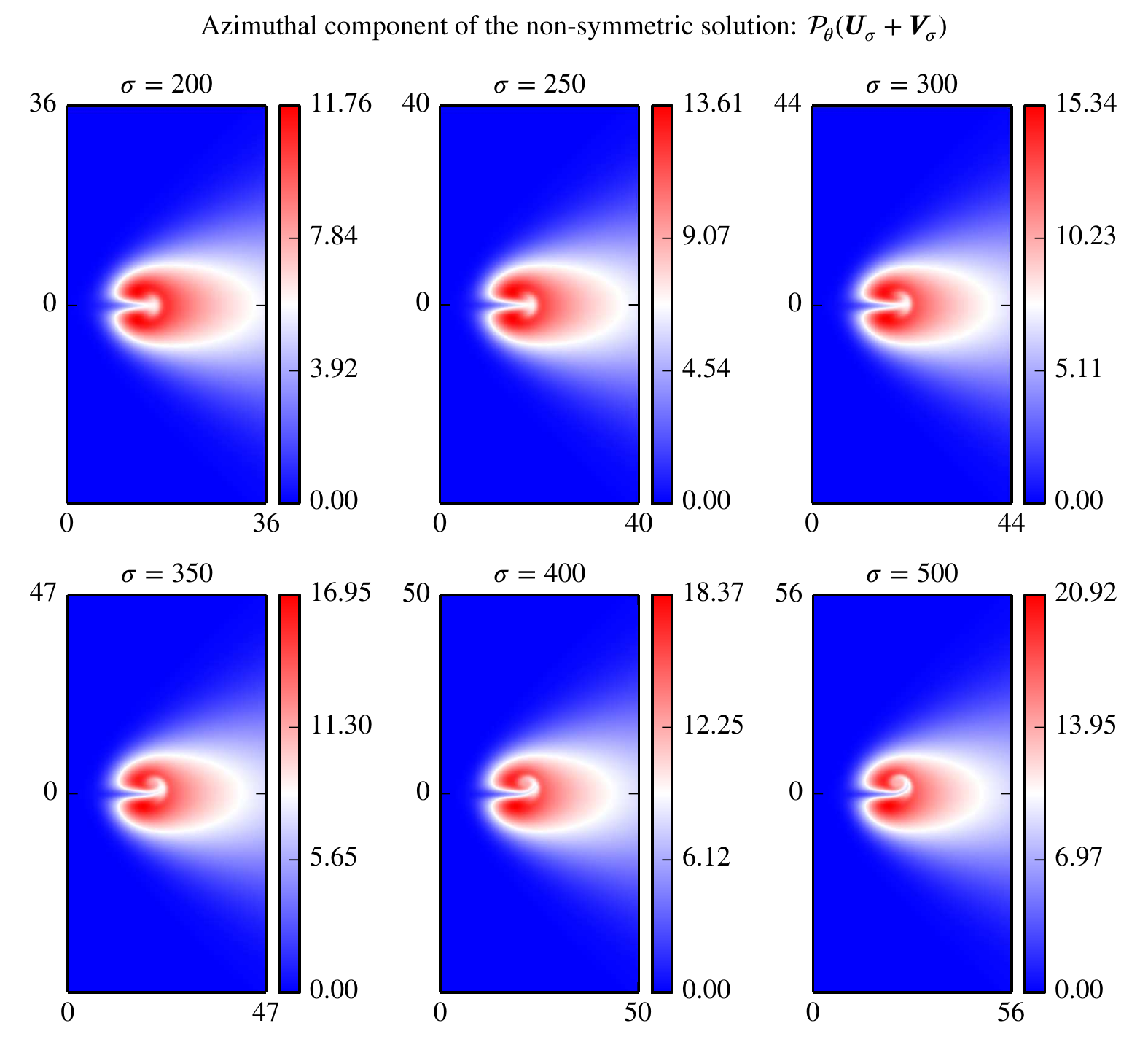}\caption{\label{fig:bif_ut}Azimuthal component of the numerical solution $\bUsigma+\bVsigma$.}
\end{figure}

\begin{figure}[p]
\includefigure{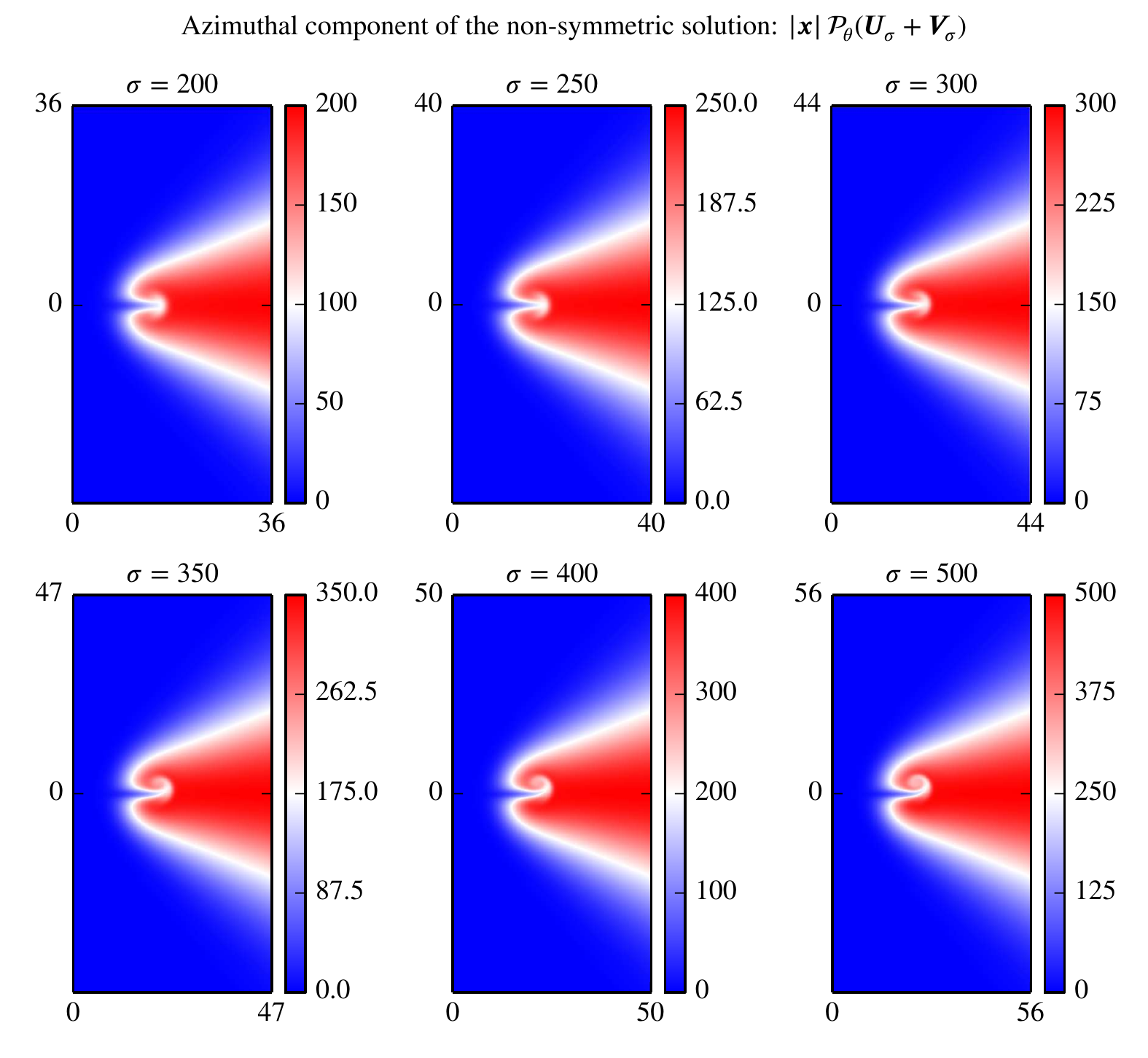}\caption{\label{fig:bif_rut}Azimuthal component of the numerical solution
$\bUsigma+\bVsigma$ multiplied by $|\bx|$.
The symmetry $\mathcal{R}$ with respect to the plane $z=0$ is broken
for $\sigma\gtrsim300$.}
\end{figure}

\begin{figure}[p]
\includefigure{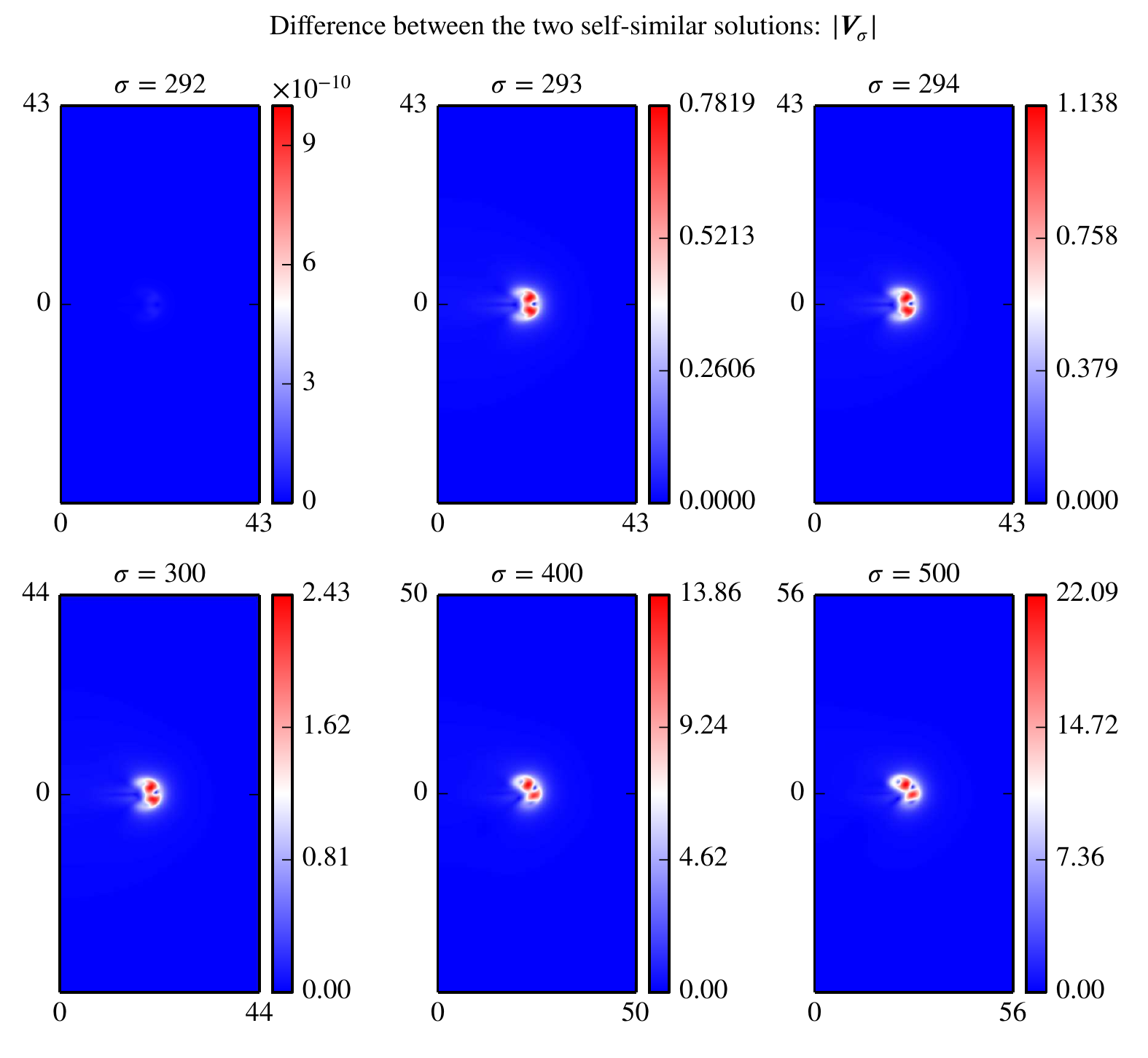}\caption{\label{fig:bif_diff}Difference between the symmetric solution $\bUsigma$
and the asymmetric solution $\bUsigma+\bVsigma$.
Before the bifurcation, \emph{i.e.} for $\sigma\lesssim292$, $\bVsigma=\bzero$,
so that both solutions coincide. After the bifurcation, the two solutions
are more and more different as $\sigma$ increases.}
\end{figure}

\begin{figure}[p]
\includefigure{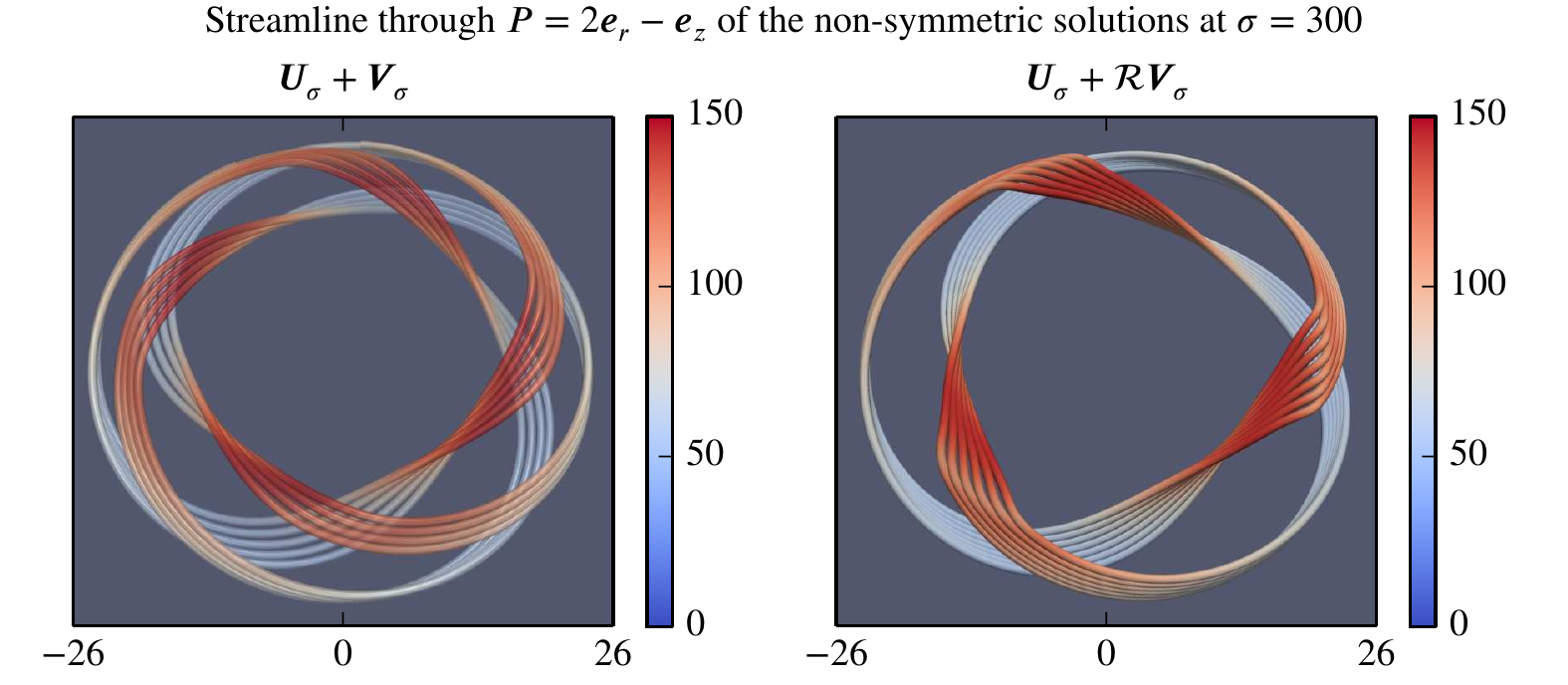}\caption{\label{fig:stream_bif}Streamline of the two non-symmetric solutions
$\bUsigma+\bVsigma$ and $\bUsigma+\mathcal{R}\bVsigma$
at $\sigma=300$ going through the point $r=2$ and $z=-1$. (a) the
streamline is almost $\frac{2\pi}{4}$-periodic; (b) the streamline
is almost $\frac{2\pi}{3}$-periodic.}
\end{figure}

\section{\label{sec:spectrum-LU}Spectrum of $\cL(\bU)$}

First we determine the point spectrum of $\cL(\bzero)$:
\begin{prop}
\label{prop:point-spectrum}The point spectrum of $\cL(\bzero)$
with domain $\mathcal{D}$ is given by a continuous part $\bigl\{\lambda\in\mathbb{C}:\Re\lambda>\frac{3}{4}\bigr\}$
and a discrete part $\bigl\{\frac{3}{2}+n\,,\:n\in\mathbb{N}\bigr\}$.
The eigenvectors of the continuous part decay like $\left|\bx\right|^{-2\lambda}$
at infinity, whereas the discrete part is characterized by eigenvectors
decaying exponentially fast like $\e^{-\left|\bx\right|^{2}/4}$.
The multiplicity of $\frac{3}{2}+n$ is $(n+1)(n+3)$ in $\mathcal{D}$
and $n+1$ in $\mathcal{D}_{\axi}$.\end{prop}
\begin{proof}
The point spectrum of $\cL(\bzero)$ is characterized by
\begin{align}
-\Delta\bv-\frac{\bx}{2}\bcdot\bnabla\bv-\frac{1}{2}\bv+\bnabla p & =\lambda\bv\,, & \bnabla\bcdot\bv & =0\,,\label{eq:eigen-problem-sigma0}
\end{align}
so by taking the divergence of the equation, we get $\Delta p=0$,
and we can choose $p=0$. Since $\bv$ is divergence-free, we use
the poloidal-toroidal decomposition,
\[
\bv=\mathbf{T}(\psi)+\mathbf{S}(\phi)\,,
\]
where $\psi$ and $\phi$ are two scalar fields and
\begin{align*}
\mathbf{T}(\psi) & =\bnabla\bwedge(\psi\bx)\,, & \mathbf{S}(\phi) & =\bnabla\bwedge\mathbf{T}(\phi)\,.
\end{align*}
Since
\begin{align*}
\Delta\mathbf{T}(\psi) & =\mathbf{T}(\Delta\psi)\,, & \Delta\mathbf{S}(\phi) & =\mathbf{S}(\Delta\phi)\,,
\end{align*}
and
\begin{align*}
\bx\bcdot\bnabla\mathbf{T}(\psi) & =\mathbf{T}(\bx\bcdot\bnabla\psi)\,, & \bx\bcdot\bnabla\mathbf{S}(\psi) & =\mathbf{S}(\bx\bcdot\bnabla\phi)-\mathbf{S}(\phi)\,,
\end{align*}
we obtain that \eqref{eigen-problem-sigma0} is transformed into
\begin{align*}
-\Delta\psi-\frac{\bx}{2}\bcdot\bnabla\psi & =\left(\lambda+\frac{1}{2}\right)\psi\,, & -\Delta\phi-\frac{\bx}{2}\bcdot\bnabla\phi & =\lambda\phi\,.
\end{align*}
Both equations being similar, we focus on the second one. Due to the
spherical symmetric, the separation of variables can be used in spherical
coordinates $(r,\theta,\varphi)$ and the eigenvectors are given by
\begin{align*}
\phi_{\lambda lm}(r,\theta,\varphi) & =f_{\lambda l}(r)Y_{lm}(\theta,\varphi)\,,
\end{align*}
where $Y_{lm}$ are the spherical harmonics and $f_{\lambda l}$ satisfies
the following radial equation
\begin{equation}
f_{\lambda l}^{\prime\prime}+\left(\frac{2}{r}+\frac{r}{2}\right)f_{\lambda l}^{\prime}-\frac{l(l+1)}{r^{2}}f_{\lambda l}+\lambda f_{\lambda l}=0\,.\label{eq:ode}
\end{equation}
In the above, $l\in\mathbb{N}$ and $m\in\{-l,-l+1,\dots,l-1,l\}$.
Explicitly, we have
\begin{align*}
\mathbf{T}(\phi_{\lambda lm}) & =f_{\lambda l}(r)\left(\csc\theta\,\partial_{\varphi}Y_{lm}(\theta,\varphi)\be_{\theta}-\partial_{\theta}Y_{lm}(\theta,\varphi)\be_{\varphi}\right)\,,\\
\mathbf{S}(\phi_{\lambda lm}) & =l(l+1)\frac{f_{\lambda l}(r)}{r}Y_{lm}(\theta,\varphi)\be_{r}+\frac{\left(rf_{\lambda l}(r)\right)^{\prime}}{r}\left(\partial_{\theta}Y_{lm}(\theta,\varphi)\be_{\theta}+\csc\theta\,\partial_{\varphi}Y_{lm}(\theta,\varphi)\be_{\varphi}\right),
\end{align*}
and one can check that the unique solution of \eqref{ode} leading
to continuous and nontrivial fields requires $l\geq1$ and is given
by
\[
f_{\lambda l}(r)=r^{l}\,\hypF\left(\lambda+\frac{l}{2};l+\frac{3}{2};\frac{-r^{2}}{4}\right)\,,
\]
where $\hypF$ is the Kummer's confluent hypergeometric function.
At large values of $r$, we have
\[
f_{\lambda l}(r)=\left(1+O(r^{-2})\right)\begin{cases}
\e^{-r^{2}/4}r^{2\lambda-3}\,, & \text{if}\;\lambda-\frac{3+l}{2}\in\mathbb{N}\,,\\
r^{-2\lambda}\,, & \text{otherwise}\,.
\end{cases}
\]
Therefore, the spectrum of \eqref{eigen-problem-sigma0} in $\mathcal{V}$
has a continuous part $\Re\lambda>\frac{3}{4}$ and a discrete part
given by $\lambda_{n}=\frac{3}{2}+\frac{n}{2}$ for $n\in\mathbb{N}$,
characterized by eigenvectors decaying exponentially fast at infinity.
The eigenspace corresponding to $\lambda_{n}$ for $n\in2\mathbb{N}$
is span by $\mathbf{T}(\phi_{(\lambda+1/2)lm})$ with $l\in\{1,3,5,\dots,n+1\}$
and by $\mathbf{S}(\phi_{\lambda lm})$ with $l\in\{2,4,6,\dots,n\}$,
where in both $m\in\{-l,-l-1,\dots,l-1,l\}$. For $n\in2\mathbb{N}+1$,
the eigenspace corresponding to $\lambda_{n}$ is span by $\mathbf{T}(\phi_{(\lambda+1/2)lm})$
with $l\in\{2,4,6,\dots,n+1\}$ and by $\mathbf{S}(\phi_{\lambda lm})$
with $l\in\{1,3,5,\dots,n\}$, always with $m\in\{-l,-l-1,\dots,l-1,l\}$.
Hence the multiplicity of $\lambda_{n}$ is $(n+1)(n+3)$. In $\mathcal{V}_{\axi}$
the eigenvectors are characterized by $m=0$, so the multiplicity
of $\lambda_{n}$ is $n+1$.
\end{proof}
Using \propref{point-spectrum}, the proof of \thmref{spectrum-L}
follows by applying results by \citet{Gallay.Wayne-Invariantmanifoldsand2002}
and \citet{Jia-Areincompressible3d2015}:
\begin{proof}[Proof of \thmref{spectrum-L}.]
The spectrum of the operator $L\bv=-\Delta\bv-\frac{\bx}{2}\bcdot\bnabla\bv-\frac{1}{2}\bv$
on domain $L^{2}(\mathbb{R}^{3})$ without divergence-free condition,
was determined explicitly by \citet[Theorem A.1]{Gallay.Wayne-Invariantmanifoldsand2002},
\[
\sigma(L)=\bigl\{\lambda\in\mathbb{C}:\Re\lambda\geq\tfrac{3}{4}\bigr\}\cup\bigl\{\tfrac{3}{2}+n\,,\;n\in\mathbb{N}\bigr\}\,.
\]
Therefore we directly obtain that $\sigma(\cL(\bzero))\subset\sigma(L)$.
The fact that the spectrum of $\sigma(\cL(\bzero))$ coincide
with the spectrum of $\sigma(L)$ follows from \propref{point-spectrum}.

Since the operator $\cL(\bU)-\cL(\bzero)$ is a relatively
compact perturbation of $\cL(\bzero)$, the essential spectrum
is unchanged, and \eqref{spectrum-LU} follows, see \citet[Lemma 2.7]{Jia-Areincompressible3d2015}.
\end{proof}

\section{\label{sec:bifurcation}Continuation and bifurcation}

In this section, we sketch the proof of \thmref{continuation-bifurcation},
since it follows by applying standard results from the theory of bifurcations:
\begin{proof}[Proof of \thmref{continuation-bifurcation}.]
First of all, since $\mathcal{D}\subset W^{2,4}(\mathbb{R}^{3})\subset C^{1}(\mathbb{R}^{3})$
with continuous embeddings, we directly deduce the continuity of the
map $F\colon\mathcal{D}\times\mathbb{R}\to\mathcal{V}$ defined by
\eqref{def-F}. Therefore, $F$ is smooth since it is quadratic.

For the first part, since $0\notin\sigma\bigl(\mathcal{L}(\bUzero)\bigr)$,
then $\mathcal{L}(\bUzero)=D_{1}F(\bv_{0},\sigma_{0})$ is invertible, so the result follows
by applying the implicit function theorem \citep[\S I.1]{Kielhoefer-bifurcation2012}.

For the second part, we define $\hat{F}(\bw,\sigma)=F(\bv_{1}(\sigma)+\bw,\sigma)$,
so that $\hat{F}(\bzero,\sigma)=\bzero$. The aim is to find a nontrivial
solution of $\hat{F}(\bw,\sigma)=\bzero$. Since $\bv_{1}:(\sigma_{0}-\varepsilon,\sigma_{0}+\varepsilon)\to\mathcal{D}$
is smooth, we deduce the smoothness of $\hat{F}:\mathcal{D}\times\mathbb{R}\to\mathcal{V}$.
We have $D_{1}\hat{F}(\bzero,\sigma_{0})=D_{1}F(\bv_{0},\sigma_{0})=\mathcal{L}(\bUzero)$,
so
\begin{align*}
\Kernel\bigl(D_{1}\hat{F}(\bzero,\sigma_{0})\bigr) & =\operatorname{span}(\bphi)\,, & \Range\bigl(D_{1}\hat{F}(\bzero,\sigma_{0})\bigr) & =\Range\bigl(\mathcal{L}(\bUzero)\bigr)\,.
\end{align*}
Since $\mathcal{L}(\bUzero)-\mathcal{L}(\bzero)$ is a relatively
compact perturbation of $\mathcal{L}(\bzero)$, $\mathcal{L}(\bUzero)$
is a Fredholm operator of index zero, hence $\hat{F}(\bcdot,0)$ is
a Fredholm operator of index zero. Moreover,
\begin{align*}
D_{12}\hat{F}(\bzero,\sigma_{0})\bphi & =D_{12}F(\bv_{0},\sigma_{0})\bphi+D_{11}F(\bv_{0},\sigma_{0})(\bphi,\bv_{1}^{\prime}(\sigma_{0}))\\
 & =\left(\bA_{0}+\bv_{1}^{\prime}(\sigma_{0})\right)\bcdot\bnabla\bphi+\bphi\bcdot\bnabla\left(\bA_{0}+\bv_{1}^{\prime}(\sigma_{0})\right)+\bnabla p\,,\\
 & =\bpsi\bcdot\bnabla\bphi+\bphi\bcdot\bnabla\bpsi+\bnabla p\,,
\end{align*}
where $\bpsi=\bA_{0}+\bv_{1}^{\prime}(\sigma_{0})=\partial_{\sigma}\bUsigma\bigl|_{\sigma=\sigma_{0}}$
so by hypothesis $D_{12}\hat{F}(\bzero,\sigma_{0})\bphi\notin\Range\bigl(D_{1}\hat{F}(\bzero,\sigma_{0})\bigr)$.
Therefore, we can apply the Crandall–Rabinowitz theorem
stated in \citet[Theorem I.5.1]{Kielhoefer-bifurcation2012} to obtain
a nontrivial smooth curve $\bigl\{\bigl(\bw(s),\sigma_{2}(s)\bigr)\in\mathcal{D}\times\mathbb{R}\,,\;s\in(-\varepsilon,\varepsilon)\bigr\}$
through $(\bv_{0},\sigma_{0})$ such that $\hat{F}(\bw(s),\sigma_{2}(s))=\bzero$,
$\bw(0)=\bzero$ and $\sigma_{2}(0)=\sigma_{0}$. Then by defining
$\bv_{2}(s)=\bv_{1}(\sigma_{2}(s))+\bw(s)$, we obtain that $\bigl\{\bigl(\bv_{2}(s),\sigma_{2}(s)\bigr)\in\mathcal{D}\times\mathbb{R}\,,\;s\in(-\varepsilon,\varepsilon)\bigr\}$
is a smooth solution curve through $(\bv_{0},\sigma_{0})$ such that
$F(\bv_{2}(s),\sigma_{2}(s))=\bzero$, $\bv_{2}(0)=\bv_{0}$ and $\sigma_{2}(0)=\sigma_{0}$.
Since $D_{11}\hat{F}(\bzero,\sigma_{0})=D_{11}F(\bv_{0},\sigma_{0})$,
we have
\[
D_{11}\hat{F}(\bzero,\sigma_{0})(\bphi,\bphi)=2\bphi\bcdot\bnabla\bphi+\bnabla p\,,
\]
and the nature of the bifurcation follows from the discussion in \citet[\S I.6]{Kielhoefer-bifurcation2012}. We note that the usual non-degeneracy condition for the pitchfork bifurcation may not be satisfied as $\hat F$ is quadratic. On the other hand, the reflection symmetry forces the bifurcation to be of the pitchfork type.
\end{proof}

\section{\label{sec:localization}Localization of self-similar solutions}

In this section, we follow the ideas of \citet{Jia-Areincompressible3d2015}
to obtain solutions with finite energy by truncation of scale-invariant
solutions. The space $X_{T}$ is defined as
\begin{equation}
X_{T}=\bigl\{\bw\in L^{\infty}(0,T;L^{4}(\mathbb{R}^{3}))\,:\:\sup_{t\in(0,T)}t^{1/2}\left\Vert \bnabla\bw(t,\cdot)\right\Vert _{L^{4}(\mathbb{R}^{3})}<\infty\bigr\}\,,\label{eq:def-X_T}
\end{equation}
equipped with the norm
\[
\left\Vert \bw\right\Vert _{X_{T}}=\sup_{t\in(0,T)}\left(\left\Vert \bw(t,\cdot)\right\Vert _{L^{4}(\mathbb{R}^{3})}+t^{1/2}\left\Vert \bnabla\bw(t,\cdot)\right\Vert _{L^{4}(\mathbb{R}^{3})}\right)\,.
\]
The space $X_{T,\axi}$ is the subspace of axi-symmetric vector fields
in $X_{T}$.

By restricting all the spaces to axi-symmetric vector fields, the result
of \citet[Theorem 1.2]{Jia-Areincompressible3d2015} becomes:
\begin{thm}
\label{thm:ns-singular}Let $\bU\in\mathcal{U}_{\axi}$ be such
that the spectrum of $\mathcal{L}(\bU)$ with domain $\mathcal{D}_{\axi}$
is included in $\bigl\{z\in\mathbb{C}\,:\:\Re z>-\beta\bigr\}$ some $\beta<\frac{1}{8}$.
Let $\bV\in\mathcal{V}_{\axi}$ be such that $\left\Vert\bV\right\Vert _{\mathcal{V}}+\left\Vert\bnabla\bV\right\Vert_{\mathcal{V}}$
is sufficiently small depending on $\left\Vert\bU\right\Vert_{\mathcal{U}}$
and $\beta$. Let
\[
\bu(t,\bx)=\frac{1}{t^{1/2}}(\bU+\bV)\biggl(\frac{\bx}{t^{1/2}}\biggr)\,.
\]
Let $\bw_{0}\in L_{\axi}^{4}(\mathbb{R}^{3})$ be a divergence-free
vector field. Then there exists a time $T>0$ and a unique solution
$\bw\in X_{T,\axi}$ to the generalized Navier–Stokes
system with singular lower order terms,
\begin{align}
\partial_{t}\bw+\bu\bcdot\bnabla\bw+\bw\bcdot\bnabla\bu+\bw\bcdot\bnabla\bw & =\Delta\bw-\bnabla p\,, & \bnabla\bcdot\bw & =0\,, & \bw(0,\cdot) & =\bw_{0}\,.\label{eq:ns-singular}
\end{align}
Here the initial condition is satisfied in the sense that
\[
\lim_{t\to0^{+}}\left\Vert \bw(t,\cdot)-\bw_{0}\right\Vert _{L^{4}(\mathbb{R}^{3})}=0\,.
\]

\end{thm}
Be using this theorem, we follow the arguments of \citet[\S 5]{Jia-Areincompressible3d2015}
to localize the two self-similar solutions to $L^{2}(\mathbb{R}^{3})$:
\begin{proof}[Proof of \thmref{localization}.]
By assuming that our numerical results reflect the actual behavior of the solutions,
we obtain the existence of two different axi-symmetric self-similar solutions
$\bUsigma$ and $\bUsigma+\bVsigma$ for $\sigma>\sigma_{0}$ satisfying
\eqref{ns-scale} with the same initial datum $\bu_{0}=\sigma\ba_{0}$,
\begin{align*}
\bu_{1}(t,\bx) & =\frac{1}{t^{1/2}}\bUsigma\biggl(\frac{\bx}{t^{1/2}}\biggr)\,, & \bu_{2}(t,\bx) & =\frac{1}{t^{1/2}}\left(\bUsigma+\bVsigma\right)\biggl(\frac{\bx}{t^{1/2}}\biggr)\,.
\end{align*}
By choosing $\sigma>\sigma_{0}$ close enough to $\sigma_{0}$, we can assume that the crossing eigenvalue $\lambda_{\sigma}$ in \eqref{numerics-spectrum} satisfies $\lambda_{\sigma}>-\frac{1}{8}$ and moreover, we can make
$\left\Vert\bVsigma\right\Vert_{\mathcal{V}}+\left\Vert\bnabla\bVsigma\right\Vert_{\mathcal{V}}$
small enough to apply \thmref{ns-singular}. By a cutoff of the stream
function associated to $\bu_{0}$, we can write
$\bu_{0}=\tilde{\bu}_{0}-\bw_{0}$,
where $\tilde{\bu}_{0}$ is a divergence-free vector field of compact
support in $B_{2R}$ and equal to $\bu_{0}$ on $B_{R}$ and $\bw_{0}$
is a divergence-free vector field such that $\left\Vert \bw_{0}\right\Vert _{L^{4}(\mathbb{R}^{3})}\leq CR^{-1/4}$.
By taking $R$ large enough, we can apply \thmref{ns-singular} with
initial data $\bw_{0}$, $\bU=\bUsigma$, and $\bV=\bzero$, to obtain
a solution $\bw_{1}\in X_{T,\axi}$ of \eqref{ns-singular}. Therefore
$\tilde{\bu}_{1}=\bu_{1}+\bw_{1}$ is an axi-symmetric solution of
the Navier–Stokes system \eqref{ns-cauchy} with initial
data $\tilde{\bu}_{0}$. In the same way, by applying \thmref{ns-singular}
with initial data $\bw_{0}$, $\bU=\bUsigma$, and $\bV=\bVsigma$,
we obtain a solution $\bw_{2}\in X_{T,\axi}$ of \eqref{ns-singular},
so that $\tilde{\bu}_{2}=\bu_{2}+\bw_{2}$ is an axi-symmetric solution
of the Navier–Stokes system \eqref{ns-cauchy} with initial
data $\tilde{\bu}_{0}$. By the standard regularity theory of the
Navier–Stokes equations $\bw_{1},\bw_{2}\in C^{\infty}((0,T)\times\mathbb{R}^{3})$,
so by using \thmref{jia-sverak}, we obtain that $\tilde{\bu}_{1},\tilde{\bu}_{2}\in C^{\infty}((0,T)\times\mathbb{R}^{3})$.
Since $\tilde{\bu}_{0}\in L^{2}(\mathbb{R}^{3})$, one can show that
$\tilde{\bu}_{1}$ and $\tilde{\bu}_{2}$ are Leray–Hopf
solutions, for example by using the results of \citet[Lemma 2.2]{Jia-Minimal-L3-Initial2013}.

We now prove that $\tilde{\bu}_{1}$ and $\tilde{\bu}_{2}$ are not
equal. Since $\bw_{1}$ and $\bw_{2}$ are uniformly bounded in $L^{4}(\mathbb{R}^{3})$,
we see that
\begin{eqnarray*}
\left\Vert \tilde{\bu}_{1}(t,\cdot)-\tilde{\bu}_{2}(t,\cdot)\right\Vert _{L^{4}(\mathbb{R}^{3})} & \geq & \left\Vert \bu_{1}(t,\cdot)-\bu_{2}(t,\cdot)\right\Vert _{L^{4}(\mathbb{R}^{3})}-\left\Vert \bw_{1}(t,\cdot)-\bw_{2}(t,\cdot)\right\Vert _{L^{4}(\mathbb{R}^{3})}\\
 & \geq & t^{-1/8}\left\Vert \bVsigma\right\Vert _{L^{4}(\mathbb{R}^{3})}-C\,,
\end{eqnarray*}
is unbounded as $t\to0^{+}$, and therefore $\tilde{\bu}_{1}$ and
$\tilde{\bu}_{2}$ are not equal since $\bVsigma$ is not trivial.

We now prove that $\tilde{\bu}_{1}$ and $\tilde{\bu}_{2}$ belong
to the complement of Serrin class. Since $\bw_{1}\in X_{T}$, we obtain
that $\sup_{t\in(0,T)}t^{1/2}\left\Vert \bw_{1}\right\Vert _{L^{\infty}(\mathbb{R}^{3})}<\infty$,
so $\bw_{1}\in L^{p}(0,T;L^{\infty}(\mathbb{R}^{3}))$ for $1\leq p<2$.
By interpolation, we obtain $\bw_{1}\in L^{p}(0,T;L^{q}(\mathbb{R}^{3}))$
for
\begin{equation}
p=\infty\text{ and }q=4\qquad\text{or}\qquad\frac{2}{p}+\frac{4}{q}>1\text{ and }q\geq4\label{eq:reqion-w}
\end{equation}
as drawn on \figref{region}. We split the space into $\mathbb{R}^{3}=B\cup B^{c}$
where $B$ is the ball of radius one centered at the origin and $B^{c}$
its complement. By using the explicit decay \eqref{bound-U} of the
self-similar solutions, we obtain that $\bu_{1}\in L^{p}(0,T;L^{q}(B))$
for $\frac{2}{p}+\frac{3}{q}>1$ and therefore $\tilde{\bu}_{1}\in L^{p}(0,T;L^{q}(B))$
also for $\frac{2}{p}+\frac{3}{q}>1$. In the same way, we can prove
that $\bu_{1}\in L^{\infty}(0,T;L(B^{c}))$ for $q>3$, so $\tilde{\bu}_{1}\in L^{p}(0,T;L^{q}(B^{c}))$
for $p$ and $q$ satisfying \eqref{reqion-w}. Since $\tilde{\bu}_{1}\in L^{\infty}(0,T;L^{2}(\mathbb{R}^{3}))$,
by interpolation we obtain that $\tilde{\bu}_{1}\in L^{p}(0,T;L^{q}(B^{c}))$
for $\frac{2}{p}+\frac{4}{q}>1$ and $q\geq2$. Therefore we proved
that $\tilde{\bu}_{1}\in L^{p}(0,T;L^{q}(\mathbb{R}^{3}))$ for $\frac{2}{p}+\frac{3}{q}>1$
and $q\geq2$. The same procedure applies to $\tilde{\bu}_{2}$ and
the proof is finished.
\end{proof}

\subsubsection*{Acknowledgments}

The authors would like to thank J.~Gómez-Serrano, H.~Jia, and V.~Vicol for valuable discussions and comments.
Parts of this work were done while J.~Guillod was at the School of Mathematics
of the University of Minnesota, the Mathematics Department of Princeton University,
and the ICERM at Brown University. The hospitality and facilities of these
institutions are gratefully acknowledged.
The research of J.~Guillod was supported by the Swiss National Science Foundation
grants \href{http://p3.snf.ch/Project-161996}{161996}
and \href{http://p3.snf.ch/Project-171500}{171500}.
The research of V.~Šverák was partially supported by grant
\href{https://www.nsf.gov/awardsearch/showAward?AWD_ID=1362467}{DMS 1362467}
from the National Science Foundation.

\bibliographystyle{merlin-doi}
\phantomsection\addcontentsline{toc}{section}{\refname}\bibliography{biblio}

\end{document}